\documentclass[a4paper]{article}         
\usepackage[left=2.5cm,right=3cm,top=2.3cm,bottom=2cm,includeheadfoot]{geometry}

\usepackage[english]{babel}
\usepackage[utf8]{inputenc}

\usepackage{amsmath}
\usepackage{amsfonts}
\usepackage{amssymb}
\usepackage{amsthm} 
\usepackage{stmaryrd}
\usepackage{dsfont}

\usepackage[bf,hang,nooneline,font=footnotesize]{caption}						
\usepackage{graphicx}														
\usepackage{tabularx}															
\usepackage{multirow}															
\usepackage{booktabs}
\usepackage[labelformat=brace, font=scriptsize, textfont=it]{subcaption}

\theoremstyle{plain}
\newtheorem{theorem}{Theorem}[section]

\newtheorem{lemma}[theorem]{Lemma}
\newtheorem{corollary}[theorem]{Corollary}
\newtheorem{prop}[theorem]{Proposition}

\theoremstyle{definition}

\theoremstyle{remark}

\usepackage{titlesec}
\titleformat{\section}[runin]{\normalfont\bfseries}{\thesection .}{0.5em}{}
\titleformat{\subsection}[runin]{\normalfont\bfseries}{\thesubsection .}{0.5em}{}

\usepackage{url}

\usepackage{hyperref}

\usepackage{graphicx}
\usepackage{xcolor}

\usepackage{array}
\usepackage[titletoc,title]{appendix}

\usepackage{changes}

\allowdisplaybreaks

\usepackage{todonotes}

\begin{document}

	\title{
		An Irreversible Investment Problem with Incomplete Information about Profitability
	}
	
	\author{Fabian Gierens\thanks{Trier University, Department IV - Mathematics}
		\hspace{1cm}  Berenice Anne Neumann\footnotemark[1]   }
	\maketitle

	\allowdisplaybreaks

	\begin{abstract}	
		We analyze an irreversible investment decision for a project which yields a flow of future operating profits given by a geometric Brownian motion with unknown drift. 
		In contrast to similar optimal stopping problems with incomplete information, the agent's payoff now depends directly on the unknown drift and not only indirectly through the underlying dynamics. 
		Hence, many standard arguments are not applicable. 
		Nonetheless, we show that it is optimal to invest in the project if the current profit level exceeds a threshold depending on the current belief for the true state of the unknown drift.
		These thresholds are described by a boundary function, for which we establish structural properties like monotonicity and continuity. 
		To prove these, we identify a central class of stopping times with useful features. Moreover, we characterize the boundary function as the unique solution of a nonlinear integral equation. 
		Building on this characterization we compute the boundary function numerically and investigate the value of information. \\
		
		
		\noindent\textbf{MSC Codes (2020):} 93C41, 60G40, 93E11

	\end{abstract}

\section{Introduction}
In many investment scenarios, decision makers have to decide if and when to take investment opportunities, but the decision cannot be easily altered or reversed once it is made. Noting that the economic conditions are dynamic and uncertain, the flexibility to choose the exact timing of the irreversible investment is valuable as it gives the possibility to wait for more favorable economic conditions. 
Hence, to identify the optimal investment time is a central problem in economic decision-making with a close link with optimal stopping \cite{DixitPindyck1994, McDonaldSiegel1986}.

In the present paper, we consider an irreversible investment problem where an agent can decide to invest in a project which afterwards generates an infinite flow of profits given by a geometric Brownian motion with unknown drift. As usual, the agent is risk-neutral and discounts future cash flows. In the case when the dynamics of these profits are completely known, it is optimal to stop once the profit process reaches a certain threshold, which can be calculated explicitly \cite[Sec. 6.1]{DixitPindyck1994}.
Here, we consider a variant with incomplete information where the average growth rate (drift) of the profits is not known but takes one of two potential values. In this case, the agent has an additional factor to include in her analysis: Besides the option to wait for more favorable conditions, she also gains information about the true drift of the profit process while waiting. 
Therefore, understanding how learning influences the decision-making process is of great interest when the economic perspective of investment projects is not known (e.g., technological innovations). 

Drift uncertainty has been investigated for many optimal stopping as well as stochastic control problems:
D\'{e}camps, Mariotti and Villeneuve \cite{DecampsMariottiVilleneuve2005} and Klein \cite{Klein2009} 
also investigate an investment problem, but in their case the investment generates a lump-sum payment, Ekström and Lu \cite{EkströmLu2011} and Ekström and Vaicenavicius \cite{EkströmVaicenavicius2016} analyze asset liquidation problems and Gapeev \cite{Gapeev2012} and Ekström and Vannest\r{a}l \cite{EkströmVannestal2019} consider American options.
Karatzas and Zhao \cite{KaratzasZhao2001} analyze a portfolio optimization problem, De Angelis \cite{DeAngelis2019} investigates an optimal dividend problem and Dammann and Ferrari \cite{DammannFerrari2023} consider optimal liquidation under price impact. 
Additionally, stopping games with drift uncertainty have been investigated, e.g.,  \cite{DeAngelisEkströmGlover2021,DeAngelisGensbittelVilleneuve2021}.
In all these papers the agent's payoff does not depend directly on the unknown drift but only indirectly through the dynamics of the payoff process.
In contrast, the payoff in our problem depends strongly on the unknown drift since the payoff of the investment is not only determined by the current profit level but rather by the subsequent growth of the profits.

As usual, also for our investment problem the investment decision only depends on the current profit level and the current belief on the true state of the unknown drift. More precisely, the optimal strategy is to stop as soon as the current profit exceeds a threshold depending on the current belief on the true state of the unknown drift. 
These thresholds are specified by a boundary function, which monotonically and continuously connects the two thresholds for the problems with known drift.
As usual for problems with incomplete information, to prove these results, we transform our problem into a two-dimensional Markovian optimal stopping problem with complete information by introducing the belief process.
Using standard arguments, we prove that a stopping time exists and that it is characterized by the boundary function. However, for further investigations of the boundary function standard arguments from the literature are not applicable since the payoff depends directly on both the profit and the belief process.
Instead, to derive structural properties like monotonicity and continuity of the boundary function, we introduce a special class of stopping times, which allow us to connect  the problems for different initial profit levels and initial beliefs.

Moreover, as for other multi-dimensional stopping problems (e.g., \cite{ ChristensenSalminen2018, DammannFerrari2022,DetempleKitapbayev2020}), we derive a nonlinear integral equation of which the boundary function is the unique solution. This allows us to introduce an iterative scheme that we solve numerically and which we use to investigate the value of information.
In the literature, the derivation of these integral equations is tailored to the specific problem specification. Our setting differs from these examples by exhibiting both dependent underlying processes and a non-convex payoff and value function.
Nonetheless, we can adapt their arguments to our setting and in particular establish the smooth fit conditions for the associated free boundary problem using the stopping times introduced before.
Finally, to prove that the boundary function is the unique solution of this integral equation, we rely on a result by Christensen et al. \cite{ChristensenCrocceMordeckiSalminen2019} and exploit the dependency of the profit and the belief process.

The paper is organized as follows. 
In Section~\ref{sec:Model} we provide the precise formulation of our investment problem with uncertain drift. 
In Section~\ref{sec:Transformation} we provide the Markovian reformulation of the problem and introduce the special class of stopping times. 
In Section~\ref{sec:Characterization} we derive central properties of the boundary function and thereafter characterize the boundary function by 
a nonlinear integral equation. 
Finally, Section~\ref{sec:Numerics} illustrates our findings with a numerical study.
	
\section{Model}
	\label{sec:Model}
	
	An agent has the option to invest in a risky project. 
	The investment is irreversible and entails sunk cost $I > 0$. 
	Once installed, the project generates an infinite flow of profits with no operating cost. 
	The profits follow a geometric Brownian motion with unknown drift $\mu_{\theta}$ and known volatility $\sigma > 0$, i.e.,
	\begin{displaymath}
		dX_t = \mu_\theta X_t dt + \sigma X_t dW_t, \quad X_0 = x,
	\end{displaymath} 
	where  $W = (W_t)_{t \ge 0}$ is a standard Brownian motion.
	The drift is determined by a random variable $\theta$, independent of the Brownian motion $W$ and with values in $\{0,1\}$ such that $\mathbb{P}(\theta=0) = 1- \pi$ and $\mathbb{P}(\theta=1)=\pi$. 
	The agent is risk-neutral and discounts future profits and costs at a discount rate $r > 0$. 
	As usual we assume $r > \max\{\mu_0,\mu_1\}$. 
	Moreover, without loss of generality, we assume that $\mu_0<\mu_1$.
	
	The agent observes the profit flow $X$ and decides when to invest in the project.
	Formally, this means that the agent chooses an $\mathcal{F}^X$-stopping time $\tau$, where $\mathcal{F}^X$ is the filtration generated by $X$. 
	The agent is allowed to choose $\tau = \infty$, which means that the agent does not invest in the project.
	Hence, the optimization problem can be summarized as 
	\begin{align}
		\label{eq:OSP-Initial}
		V(x,\pi) 
		&:= \sup_{\tau \, \mathcal{F}^X\text{-stopping time}} \mathbb{E} \left[ e^{-r\tau} \left( \int_0^\infty e^{-rs} X_{\tau+s} \, ds - I \right) \right].
	\end{align}
		
	The above set-up is modeled by the following probability space:
	We denote by $C(\mathbb{R_+})$ the set of all continuous functions mapping from $\mathbb{R}_+$ to $\mathbb{R}$. 
	We define $\Omega = [0,1] \times C(\mathbb{R_+})$ and equip it with the Borel $\sigma$-algebra $\mathcal{B}\left([0,1] \times  C(\mathbb{R_+}) \right)$ and 
	the probability measure $\mathbb{P}$ given by $\mathbb{P} = \mathbb{U} \otimes \mathbb{W}$, where $\mathbb{U}$ is the standard uniform distribution and $\mathbb{W}$ is the Wiener measure.
	For $\pi \in [0,1]$, we define the random variable $\theta^\pi$ by $\theta^\pi(u,w) = \mathbb{I}_{[0,\pi]}(u)$ for every $(u,w) \in [0,1] \times C(\mathbb{R_+})$.
	Moreover, for $t \ge 0$ we define the random variable $W_t$ by $W_t(u,w)= w(t)$ for every $(u,w) \in [0,1] \times C(\mathbb{R_+})$, which implies that $W = (W_t)_{t \ge 0}$ is a standard Brownian motion. 
	In this setting, for a given fixed initial profit level $x \in (0,\infty)$ and  prior probability $\pi \in [0,1]$, the profits are given by $X^{x,\pi} = \left( X^{x,\pi}_t \right)_{t \ge 0}$ with
	\begin{displaymath}
		X^{x,\pi}_t := x \exp \left( \left(\mu_{\theta^{\pi}}- \tfrac{1}{2} \sigma^2 \right) t + \sigma W_t \right).
	\end{displaymath}
	The augmented filtration generated by $X^{x,\pi}$ is denoted by $\mathcal{F}^{X^{x,\pi}}$ and the set of $\mathcal{F}^{X^{x,\pi}}$-stopping times is denoted by $\mathcal{T}^{x,\pi}$. 
	
	As in the case with full information (see \cite[Sec. 6.1]{DixitPindyck1994}), we can rewrite \eqref{eq:OSP-Initial} in terms of the current profits: 
	Let us fix $(x, \pi) \in (0,\infty) \times [0,1]$.	
	For $i \in \{0,1\}$, we have that $X^{x,i}$ is a geometric Brownian motion with drift $\mu_i$, which, by \cite[Sec. 6.1]{DixitPindyck1994}, implies that 
	\begin{displaymath}
		\mathbb{E} \left[ e^{-rt} \left( \int_0^\infty e^{-rs} X^{x,i}_{t+s} \, ds - I \right) \right] 
		= \mathbb{E} \left[ e^{-rt} \left( \frac{X^{x,i}_t}{r-\mu_i} - I \right) \right].
	\end{displaymath} 
	Since we can decompose $X^{x,\pi}$ as 
	\begin{align}
		\label{eq:XDecomposition}
		X^{x,\pi}_t = \mathbb{I}_{\{\theta^\pi = 0\}} X^{x,0}_t + \mathbb{I}_{\{\theta^\pi = 1\}} X^{x,1}_t, \quad t \ge 0,
	\end{align}
	we obtain for every  $t \ge 0$ that
	\begin{align}
		\label{eq:Derivation_ValueFunctionWithoutIntegral}
		\mathbb{E} \left[ e^{-rt} \left( \int_0^\infty e^{-rs} X^{x,\pi}_{t+s} \, ds - I \right) \right] 
		= \mathbb{E} \left[ e^{-rt} \left( \frac{X^{x,\pi}_t}{r-\mu_{\theta^\pi}} - I \right) \right]
	\end{align}
	and we arrive at
	\begin{equation}
		\label{eq:OSP-Start}
		V(x,\pi) = \sup_{\tau \in \mathcal{T}^{x,\pi}} \mathbb{E} \left[ e^{-r\tau} \left( \frac{X^{x,\pi}_\tau}{r-\mu_{\theta^\pi}} - I \right) \right].
	\end{equation}
	
\section{A Markovian Reformulation with Full Information}
	\label{sec:Transformation}
	
	In this section we prove that there is an optimal stopping time for the problem \eqref{eq:OSP-Start}. 
	Noting that the random variable $\theta^\pi$ is not observable, we face an optimal stopping problem with incomplete information, where the profit process $X^{x,\pi}$ conveys some (noisy) information about $\mu_{\theta^\pi}$. 
	As usual, we solve this problem by introducing a belief process, which summarizes the current knowledge about the unknown random variable $\theta^\pi$. 
	Applying filtering theory, we then transform the problem into a two-dimensional Markovian stopping problem with full information. 
	Using standard results, we obtain that a first entry time for this two-dimensional process is optimal. 
	Thereafter, we establish that the current belief is a function of the initial values $(x, \pi)$, the current time and the current level of profits. 
	This insight allows us to focus on a particular class of stopping times in the subsequent analysis. 
	As a preparation for the next section, we derive first structural properties of these stopping times and the value function.
	
	Let us fix the initial values $(x, \pi) \in (0, \infty) \times [0,1]$.
	As usual (e.g., \cite[Sec. 3.1]{DecampsMariottiVilleneuve2005}, \cite[p. 4]{EkströmLu2011}) we introduce the belief process $\Pi^\pi = (\Pi^\pi_t)_{t \ge 0}$ via
	\begin{displaymath}
		\Pi^\pi_t := \mathbb{P}(\theta^\pi=1|\mathcal{F}^{X^{x,\pi}}_t), \quad \Pi^\pi_0 := \pi,
	\end{displaymath} 
	which summarizes the current knowledge of the agent on $\theta^\pi$ up to time $t$. 
	By \cite[Theorems 7.12 and 9.1]{LiptserShiryaev1977} the two-dimensional process $(X^{x,\pi},\Pi^\pi)$ has the dynamics
	\begin{align}
		\label{eq:XPiDynamics}
		\begin{pmatrix} dX^{x,\pi}_t \\ d\Pi^\pi_t \end{pmatrix} 
		= \begin{pmatrix} \left[ \mu_0 + \Pi^\pi_t (\mu_1 - \mu_0) \right] X^{x,\pi}_t \\ 0 \end{pmatrix} \, dt
		+ \begin{pmatrix} \sigma X^{x,\pi}_t \\ \frac{\mu_1-\mu_0}{\sigma} \Pi^\pi_t (1-\Pi^\pi_t) \end{pmatrix} \, d\bar{W}^\pi_t,
	\end{align}
	where $\bar{W}^\pi$ is an $(\mathbb{P},\mathcal{F}^{X^{x,\pi}})$-Brownian motion. 
	By \cite[Lemma 3.1]{DecampsMariottiVilleneuve2005}, the belief process $\Pi^\pi$ lies in $[0,1]$~a.s., for $\pi \in (0,1)$ we have that $\inf \{t \ge 0: \Pi^\pi_t \notin (0,1)\} = \infty$~a.s.\ and for $\pi \in \{0,1\}$ we have $\Pi^0_t=0$ and $\Pi^1_t=1$ for all $t \ge 0$~a.s.\
	Additionally, using Lévy's upward theorem and $\mu_{\theta^\pi} - \frac{1}{2} \sigma^2  = \lim_{t \to \infty} \frac{\ln(X^{x,\pi}_t/x)}{t}$~a.s., we observe that $\Pi^\pi_t  \to \mathbb{P}(\theta^\pi=1|\mathcal{F}^{X^{x,\pi}}_\infty) = \theta^\pi $ a.s.\ for $t \to \infty$. 
	Using the belief process, we can rewrite the expected payoff associated to a stopping time $\tau \in \mathcal{T}^{x,\pi}$ as 
	\begin{align*}
		&\mathbb{E} \left[ e^{-r\tau} \left( \frac{X^{x,\pi}_\tau}{r-\mu_{\theta^\pi}} - I \right) \right] = \mathbb{E} \left[ \mathbb{E} \left[ e^{-r\tau} \left( \frac{X^{x,\pi}_\tau}{r-\mu_{\theta^\pi}} - I \right) \middle| \mathcal{F}^{X^{x,\pi}}_\tau \right] \right] \\
		&= \mathbb{E} \left[ e^{-r\tau} \left( \frac{X^{x,\pi}_\tau}{r-\mu_0} - I \right) \mathbb{P}(\theta^\pi = 0 | \mathcal{F}^{X^{x,\pi}}_\tau) + e^{-r\tau} \left( \frac{X_\tau^{x,\pi}}{r-\mu_1} - I \right)  \mathbb{P}(\theta^\pi = 1 | \mathcal{F}^{X^{x,\pi}}_\tau) \right] \\
		&= \mathbb{E} \left[ e^{-r\tau} \left( X^{x,\pi}_\tau \left( \frac{1-\Pi^\pi_\tau}{r-\mu_0} + \frac{\Pi^\pi_\tau}{r-\mu_1} \right) -I \right) \right].
	\end{align*} 
	Since $\tau$ is arbitrary, we obtain
	\begin{equation}
		\label{eq:OSP-FullInfo}
		V(x,\pi) = \sup_{\tau \in \mathcal{T}^{x,\pi}} \mathbb{E} \left[ e^{-r\tau} g(X^{x,\pi}_\tau,\Pi^\pi_\tau) \right],
	\end{equation}
	where the payoff function $g : (0,\infty) \times [0,1] \to \mathbb{R}$ is defined by
	\begin{displaymath}
		(\hat{x}, \hat{\pi}) \mapsto g(\hat{x},\hat{\pi}) := \hat{x} \left( \frac{1-\hat{\pi}}{r-\mu_0} + \frac{\hat{\pi}}{r-\mu_1} \right) - I.
	\end{displaymath} 
	The reformulation \eqref{eq:OSP-FullInfo} is now a standard two-dimensional Markovian optimal stopping problem with full information. 
	Hence, we obtain the following result using standard theory (the proof can be found in the Appendix~\ref{sec:appendix}):
	
	\begin{prop}
		\label{prop:existence_first_entry}
		The first entry time 
		$\tau_*(x,\pi) := \inf\left\{ t \ge 0: (X^{x,\pi}_t,\Pi^{\pi}_t) \in D \right\}$ with 
		\begin{align}
			\label{eq:StoppingSetD}
			D := \{(\hat{x},\hat{\pi}) \in (0,\infty) \times [0,1] : V(\hat{x},\hat{\pi}) = g(\hat{x},\hat{\pi}) \}
		\end{align} 
		is optimal for the optimal stopping problem \eqref{eq:OSP-FullInfo}.
	\end{prop}

	Noting that the process $S_t^\pi:= \ln \left( X_t^{x, \pi}/x \right) = \left( \mu_{\theta^\pi}-\frac{1}{2}\sigma^2 \right) t + \sigma W_t$ is a Brownian motion with uncertain drift, we obtain the following simple representation of the current belief $\Pi^{\pi}_t$ using the Wonham filter \cite[Eq. (5)]{Wonham1964} (the proof can be found in Appendix~\ref{sec:appendix}):
	
	\begin{prop}
		\label{prop:PiExplicitSolution}
		We have $\Pi^\pi_t = f(t,x,\pi,X^{x,\pi}_t)$ for all $t \ge 0$, where $f: [0, \infty) \times (0, \infty) \times[0,1] \times (0,\infty) \rightarrow [0,1]$ is given by
		\begin{displaymath}
			(t,\hat{x},\hat{\pi},y) \mapsto f(t,\hat{x},\hat{\pi},y) := \frac{ \hat{\pi} \left( \frac{y}{\hat{x}} e^{\frac{1}{2}(\sigma^2-\mu_1-\mu_0)t} \right)^{\frac{\mu_1-\mu_0}{\sigma^2}} }{ 1-\hat{\pi} + \hat{\pi} \left( \frac{y}{\hat{x}} e^{\frac{1}{2}(\sigma^2-\mu_1-\mu_0)t} \right)^{\frac{\mu_1-\mu_0}{\sigma^2}} }.
		\end{displaymath}
	\end{prop}
	
	In light of the last result, we can rewrite the stopping time $\tau_*(x,\pi)$ from Proposition~\ref{prop:existence_first_entry} as
	\begin{align*}
		\tau_*(x, \pi) = \inf\left\{ t \ge 0 : \left(X^{x,\pi}_t, f(t,x,\pi,X^{x,\pi}_t) \right) \in D\right\}.
	\end{align*}
	This representation is the cornerstone of our forthcoming analysis. 
	Namely, we will focus on stopping times of the form
	\begin{equation}
		\label{eq:Form_Stopping_Times}
		\tau^{x,\pi}_{\hat{\pi}} := \inf\left\{ t \ge 0 : \left(X^{x,\pi}_t, f(t,x,\hat{\pi},X^{x,\pi}_t) \right) \in D\right\},
	\end{equation}
	where $\hat{\pi} \in [0,1]$. 
	As a first step we note that these stopping times have some useful properties: 
	For any $\hat{\pi} \in [0,1]$ the stopping time $\tau^{x,\pi}_{\hat{\pi}}$ is an $\mathcal{F}^{X^{x,\pi}}$-stopping time and we have $\tau^{x,\pi}_{\pi} = \tau_*(x,\pi)$. 
	Moreover, we can decompose these stopping times as
	\begin{equation}
		\label{eq:TauDecomposition}
		\tau^{x,\pi}_{\hat{\pi}} =  \mathbb{I}_{\{\theta^\pi = 0\}} \tau^{x,0}_{\hat{\pi}} + \mathbb{I}_{\{\theta^\pi = 1\}} \tau^{x,1}_{\hat{\pi}}.
	\end{equation}
	Given $\hat{x} \in (0,\infty)$, we have $X^{\hat{x},\pi}_t = \frac{\hat{x}}{x} X^{x,\pi}_t$ for all $t \ge 0$~a.s.\ and, hence, $\tau^{\hat{x},\pi}_{\hat{\pi}}$ is also an $\mathcal{F}^{X^{x,\pi}}$-stopping time.
	Moreover, for every $\tau \in \mathcal{T}^{\hat{x},\pi}$, we obtain
	\begin{equation}
		\label{eq:xShift}
		\mathbb{E} \left[ e^{-r\tau} g(X^{\hat{x},\pi}_\tau, \Pi^\pi_\tau) \right]
		= \frac{\hat{x}}{x} \mathbb{E} \left[ e^{-r\tau} g(X^{x,\pi}_\tau,\Pi^\pi_\tau) \right] + \left( \frac{\hat{x}}{x} -1 \right) \mathbb{E} \left[ e^{-r\tau} I \right].
	\end{equation}
	
	Relying on stopping times of the form \eqref{eq:Form_Stopping_Times} and the observations made so far, we obtain the following preliminary results, which will be central for the subsequent analysis:
	
	\begin{lemma}
		\label{lemma:PropertyOfStoppingSet}
		Let  $\hat{x} \in (x,\infty)$. Then the stopping set $D$ from \eqref{eq:StoppingSetD} satisfies
		\begin{displaymath}
			(x,\pi) \in D \Rightarrow (\hat{x}, \pi) \in D.
		\end{displaymath}
	\end{lemma}
	\begin{proof}
		Assume that $(x,\pi) \in D$, i.e., we have $V(x,\pi) = g(x,\pi)$. 
		Using \eqref{eq:xShift}, we obtain
		\begin{align*}
			V(\hat{x}, \pi)
			&= \mathbb{E} \left[ e^{-r\tau^{\hat{x},\pi}_\pi} g(X^{\hat{x},\pi}_{\tau^{\hat{x},\pi}_\pi},\Pi^\pi_{\tau^{\hat{x},\pi}_\pi}) \right] 
			= \frac{\hat{x}}{x} \mathbb{E} \left[ e^{-r\tau^{\hat{x},\pi}_\pi} g(X^{x,\pi}_{\tau^{\hat{x},\pi}_\pi},\Pi^\pi_{\tau^{\hat{x},\pi}_\pi}) \right] + \left( \frac{\hat{x}}{x} -1 \right) \mathbb{E} \left[ e^{-r\tau^{\hat{x},\pi}_\pi} I \right] \\
			&\le  \frac{\hat{x}}{x} V(x,\pi) +  \left( \frac{\hat{x}}{x} -1\right) I 
			= \frac{\hat{x}}{x} g(x,\pi) +  \left( \frac{\hat{x}}{x} -1\right) I 
			= g(\hat{x}, \pi).
		\end{align*}
		Hence, we have $V(\hat{x}, \pi) = g(\hat{x}, \pi)$, which implies  $(\hat{x},\pi) \in D$.
	\end{proof}
	
	\begin{lemma}
		\label{lem:EssentialIdentity}
		For $(\hat{x},\hat{\pi}) \in (0,\infty) \times [0,1]$, we have
		\begin{displaymath}
			\mathbb{E} \left[ e^{-r\tau^{\hat{x},\pi}_{\hat{\pi}}} \left( \frac{X^{x,\pi}_{\tau^{\hat{x},\pi}_{\hat{\pi}}}}{r-\mu_{\theta^{\pi}}} - I \right) \right] = (1-\pi) \, \mathbb{E} \left[ e^{-r\tau^{\hat{x},0}_{\hat{\pi}}} \left( \frac{X^{x,0}_{\tau^{\hat{x},0}_{\hat{\pi}}}}{r-\mu_0} - I \right) \right] + \pi \, \mathbb{E} \left[ e^{-r\tau^{\hat{x},1}_{\hat{\pi}}} \left( \frac{X^{x,1}_{\tau^{\hat{x},1}_{\hat{\pi}}}}{r-\mu_1} - I \right) \right].
		\end{displaymath}
	\end{lemma}
	\begin{proof}
		Utilizing the decompositions \eqref{eq:XDecomposition} and \eqref{eq:TauDecomposition}, we obtain 
		\begin{align*}
			&\mathbb{E} \left[ e^{-r\tau^{\hat{x},\pi}_{\hat{\pi}}} \left( \frac{X^{x,\pi}_{\tau^{\hat{x},\pi}_{\hat{\pi}}}}{r-\mu_{\theta^{\pi}}} - I \right) \right]  
			=\mathbb{E} \left[  \mathbb{E} \left[ e^{-r\tau^{\hat{x},\pi}_{\hat{\pi}}} \left( \frac{X^{x,\pi}_{\tau^{\hat{x},\pi}_{\hat{\pi}}}}{r-\mu_{\theta^{\pi}}} - I \right)\middle| \theta^\pi \right] \right] \\
			&= \mathbb{P}(\theta^\pi = 0) \mathbb{E} \left[ e^{-r\tau^{\hat{x},0}_{\hat{\pi}}} \left( \frac{X^{x,0}_{\tau^{\hat{x},0}_{\hat{\pi}}}}{r-\mu_0} - I \right) \right] + \mathbb{P}(\theta^\pi = 1) \mathbb{E} \left[ e^{-r\tau^{\hat{x},1}_{\hat{\pi}}} \left( \frac{X^{x,1}_{\tau^{\hat{x},1}_{\hat{\pi}}}}{r-\mu_1} - I \right) \right]. 
		\end{align*}
		Noting that $\mathbb{P}(\theta^\pi=0) =1-\pi$ and $\mathbb{P}(\theta^\pi=1) = \pi$, we arrive at the desired representation.	
	\end{proof}

	\begin{corollary}
		\label{cor:UpperBoundValueFct}
		We have $V(x,\pi) \le (1-\pi) V(x,0) + \pi V(x,1)$.		
	\end{corollary}
	\begin{proof}
		By Lemma~\ref{lem:EssentialIdentity}, we have
		\begin{align*}
			V(x,\pi) 
			&= (1-\pi) \mathbb{E} \left[ e^{-r \tau^{x,0}_\pi} \left( \frac{X^{x,0}_{\tau^{x,0}_\pi}}{r-\mu_0} - I \right) \right] + \pi \mathbb{E} \left[ e^{-r \tau^{x,1}_\pi} \left( \frac{X^{x,1}_{\tau^{x,1}_\pi}}{r-\mu_1} - I \right) \right] \\
			&\le (1-\pi) V(x,0) + \pi V(x,1). \qedhere
		\end{align*}
	\end{proof}
		
	With these preparations, we now show that the value function is increasing and locally Lipschitz continuous in both arguments.
	\begin{lemma}
		\label{lem:ValueFunctionIncreasing}
		The value function $V$ is increasing in both arguments.
	\end{lemma}
	\begin{proof}
		Let $x,\hat{x} \in (0,\infty)$ with $x \leq \hat{x}$ and $\pi \in [0,1]$. 
		Since $X^{x,\pi}_t \leq X^{\hat{x},\pi}_t$ for all $t \ge 0$~a.s.\ and the payoff function $g$ is increasing in its first argument, we obtain
		\begin{displaymath}
			V(x,\pi) 
			= \mathbb{E} \left[ e^{-r\tau^{x,\pi}_\pi} g(X^{x,\pi}_{\tau^{x,\pi}_\pi}, \Pi^\pi_{\tau^{x,\pi}_\pi}) \right] 
			\le \mathbb{E} \left[ e^{-r\tau^{x,\pi}_\pi} g(X^{\hat{x},\pi}_{\tau^{x,\pi}_\pi}, \Pi^\pi_{\tau^{x,\pi}_\pi}) \right] 
			\le V(\hat{x},\pi).
		\end{displaymath} 
		
		Let $x \in (0,\infty)$ and $\pi, \hat{\pi} \in [0,1]$  with $\pi \le \hat{\pi}$. 
		Since $X^{x,0}_t \le X^{x,\hat{\pi}}_t$ and $\Pi_t^{\hat{\pi}} \ge 0$ for all $t \ge 0$~a.s.\ and the payoff function $g$ is increasing in its second argument, we obtain 
		\begin{displaymath}
			V(x,0) 
			= \mathbb{E} \left[ e^{-r\tau^{x,0}_0} g(X^{x,0}_{\tau^{x,0}_0}, 0) \right] 
			\le  \mathbb{E} \left[ e^{-r\tau^{x,0}_0} g(X^{x,\hat{\pi}}_{\tau^{x,0}_0}, \Pi^{\hat{\pi}}_{\tau^{x,0}_0}) \right] 
			\le V(x,\hat{\pi}).
		\end{displaymath}
		This and Lemma~\ref{lem:EssentialIdentity} imply
		\begin{align*}
			V(x,\pi) 
			&= \mathbb{E} \left[ e^{-r\tau^{x,\pi}_\pi} \left( \frac{X^{x,\pi}_{\tau^{x,\pi}_\pi}}{r-\mu_{\theta^\pi}} - I \right) \right] \\
			&= (1-\pi) \, \mathbb{E} \left[ e^{-r\tau^{x,0}_\pi} \left( \frac{X^{x,0}_{\tau^{x,0}_\pi}}{r-\mu_0} - I \right) \right] + \pi \, \mathbb{E} \left[ e^{-r{\tau^{x,1}_\pi}} \left( \frac{X^{x,1}_{\tau^{x,1}_\pi}}{r-\mu_1} - I \right) \right] \\
			&= \frac{\pi}{\hat{\pi}} \left( (1-\hat{\pi}) \, \mathbb{E} \left[ e^{-r{\tau^{x,0}_\pi}} \left( \frac{X^{x,0}_{\tau^{x,0}_\pi}}{r-\mu_0} - I \right) \right] + \hat{\pi} \, \mathbb{E} \left[ e^{-r{\tau^{x,1}_\pi}} \left( \frac{X^{x,1}_{\tau^{x,1}_\pi}}{r-\mu_1} - I \right) \right] \right) \\
			&\quad + \left( (1-\pi) - \frac{\pi}{\hat{\pi}} (1-\hat{\pi}) \right) \, \mathbb{E} \left[ e^{-r{\tau^{x,0}_\pi}} \left( \frac{X^{x,0}_{\tau^{x,0}_\pi}}{r-\mu_0} - I \right) \right] \\
			&= \frac{\pi}{\hat{\pi}} \mathbb{E} \left[ e^{-r{\tau^{x,\hat{\pi}}_\pi}} \left( \frac{X^{x,\hat{\pi}}_{\tau^{x,\hat{\pi}}_\pi}}{r-\mu_{\theta^{\hat{\pi}}}} - I \right) \right] + \left(1-\frac{\pi}{\hat{\pi}} \right) \mathbb{E} \left[ e^{-r{\tau^{x,0}_\pi}} \left( \frac{X^{x,0}_{\tau^{x,0}_\pi}}{r-\mu_0} - I \right) \right] \\
			&\le \frac{\pi}{\hat{\pi}} V(x,\hat{\pi}) + \left(1-\frac{\pi}{\hat{\pi}} \right) V(x,0)
			\le V(x, \hat{\pi}). \qedhere
		\end{align*}
	\end{proof}
	
	\begin{lemma}
		\label{lem:ValueFunctionLipschitz}
		The value function $V$ is locally Lipschitz continuous.
	\end{lemma}
	\begin{proof}
		Let $x,\hat{x} \in (0,\infty)$ with $x \leq \hat{x}$ and $\pi \in [0,1]$. Equation \eqref{eq:xShift} yields
		\begin{displaymath}
			V(\hat{x},\pi) 
			= \mathbb{E} \left[ e^{-r\tau^{\hat{x},\pi}_\pi} g(X^{\hat{x},\pi}_{\tau^{\hat{x},\pi}_\pi}, \Pi^\pi_{\tau^{\hat{x},\pi}_\pi}) \right] 
			= \frac{\hat{x}}{x} \mathbb{E} \left[ e^{-r\tau^{\hat{x},\pi}_\pi} g(X^{x,\pi}_{\tau^{\hat{x},\pi}_\pi},\Pi^\pi_{\tau^{\hat{x},\pi}_\pi}) \right] + \left( \frac{\hat{x}}{x} -1 \right) \mathbb{E} \left[ e^{-r\tau^{\hat{x},\pi}_\pi} I \right].
		\end{displaymath}
		Moreover, we have
		\begin{displaymath}
			V(x,\pi) 
			\ge \mathbb{E} \left[ e^{-r\tau^{\hat{x},\pi}_\pi} g(X^{x,\pi}_{\tau^{\hat{x},\pi}_\pi}, \Pi^\pi_{\tau^{\hat{x},\pi}_\pi}) \right].
		\end{displaymath} 
		Since $V(\hat{x},\pi) \ge V(x,\pi)$ by Lemma \ref{lem:ValueFunctionIncreasing}, we have
		\begin{align*}
			|V(\hat{x},\pi) - V(x,\pi)|
			&\le \left( \frac{\hat{x}}{x} -1 \right) \left( \mathbb{E} \left[ e^{-r\tau^{\hat{x},\pi}_\pi} g(X^{x,\pi}_{\tau^{\hat{x},\pi}_\pi},\Pi^\pi_{\tau^{\hat{x},\pi}_\pi}) \right] + \mathbb{E} \left[ e^{-r\tau^{\hat{x},\pi}_\pi} I \right] \right) \\
			&= \left( \frac{\hat{x}}{x} -1 \right) \mathbb{E} \left[ e^{-r\tau^{\hat{x},\pi}_\pi} X^{x,\pi}_{\tau^{\hat{x},\pi}_\pi} \left( \frac{1-\Pi^\pi_{\tau^{\hat{x},\pi}_\pi}}{r-\mu_0}+\frac{\Pi^\pi_{\tau^{\hat{x},\pi}_\pi}}{r-\mu_1}  \right) \right] \\
			&\le \left( \hat{x} - x \right) \frac{1}{x} \mathbb{E} \left[ 	e^{-r\tau^{\hat{x},\pi}_\pi} \frac{X^{x,1}_{\tau^{\hat{x},\pi}_\pi}}{r-\mu_1} \right] \le \frac{1}{r-\mu_1} \left| \hat{x} - x \right|,
		\end{align*}
		where the last inequality follows from the optional sampling theorem \cite[p.19, Theorem 3.22]{KaratzasShreve1998} applied to the supermartingale $\left(e^{-rt} X^{x,1}_t\right)_{t \ge 0}$ with last element $\lim_{t\to \infty} e^{-rt} X^{x,1}_t = 0$ a.s.
				
	 	Let $x \in (0,\infty)$ and $\pi,\hat{\pi} \in [0,1]$ with $\pi \leq \hat{\pi}$. By Lemma \ref{lem:EssentialIdentity}, we have
	 	\begin{align*}
	 		V(x,\hat{\pi}) 
	 		&= \mathbb{E} \left[ e^{-r\tau^{x,\hat{\pi}}_{\hat{\pi}}} \left( \frac{X^{x,\hat{\pi}}_{\tau^{x,\hat{\pi}}_{\hat{\pi}}}}{r-\mu_{\theta^{\hat{\pi}}}} - I \right) \right] \\
	 		&= (1-\hat{\pi}) \mathbb{E} \left[ e^{-r\tau^{x,0}_{\hat{\pi}}} \left( \frac{X^{x,0}_{\tau^{x,0}_{\hat{\pi}}}}{r-\mu_0} - I \right) \right] + \hat{\pi} \mathbb{E} \left[ e^{-r\tau^{x,1}_{\hat{\pi}}} \left( \frac{X^{x,1}_{\tau^{x,1}_{\hat{\pi}}}}{r-\mu_1} - I \right) \right] \\
	 		&= (1-\pi) \mathbb{E} \left[ e^{-r\tau^{x,0}_{\hat{\pi}}} \left( \frac{X^{x,0}_{\tau^{x,0}_{\hat{\pi}}}}{r-\mu_0} - I \right) \right] + \pi \mathbb{E} \left[ e^{-r\tau^{x,1}_{\hat{\pi}}} \left( \frac{X^{x,1}_{\tau^{x,1}_{\hat{\pi}}}}{r-\mu_1} - I \right) \right] \\
	 		&\quad + (\hat{\pi} - \pi) \left( \mathbb{E} \left[ e^{-r\tau^{x,1}_{\hat{\pi}}} \left( \frac{X^{x,1}_{\tau^{x,1}_{\hat{\pi}}}}{r-\mu_1} - I \right) \right] - \mathbb{E} \left[ e^{-r\tau^{x,0}_{\hat{\pi}}} \left( \frac{X^{x,0}_{\tau^{x,0}_{\hat{\pi}}}}{r-\mu_0} - I \right) \right] \right) \\
	 		&\le \mathbb{E} \left[ e^{-r\tau^{x,\pi}_{\hat{\pi}}} \left( \frac{X^{x,\pi}_{\tau^{x,\pi}_{\hat{\pi}}}}{r-\mu_{\theta^\pi}} - I \right) \right] + (\hat{\pi} - \pi)  (V(x,1)+I) \\
	 		&\le V(x,\pi) + (\hat{\pi} - \pi)( V(x,1) +I).
	 	\end{align*}
	  	Since $V$ is increasing by Lemma \ref{lem:ValueFunctionIncreasing}, we obtain
		\begin{displaymath}
			|V(x,\hat{\pi}) - V(x,\pi)|
			\le |\hat{\pi} - \pi| (V(x,1)+I). \qedhere
		\end{displaymath}
	\end{proof}

\section{Characterization of the Boundary Function}
	\label{sec:Characterization}
	We note that the existence of a boundary function $b : [0,1] \to [0,\infty]$ given by
	\begin{displaymath}
		\pi \mapsto b(\pi) := \inf\{x : (x,\pi) \in D\},
	\end{displaymath} 
	that characterizes the stopping set $D$ from~\eqref{eq:StoppingSetD} as
	\begin{displaymath}
		D = \{(x,\pi) \in (0,\infty) \times [0,1] : x \geq b(\pi) \},
	\end{displaymath}	
	immediately follows from Lemma~\ref{lemma:PropertyOfStoppingSet}.
	Intuitively, the boundary function maps the prior beliefs onto the minimal profit level, for which instantaneous stopping is optimal. 
	In this section, we will prove that the boundary function is continuous, decreasing and lies between the stopping thresholds for the problems with known drift.
	Furthermore, we derive a non-trivial lower bound for the boundary function. 
	These results allow us to establish that the boundary function is the unique solution of a nonlinear integral equation.
	
	\subsection{Properties of the Boundary Function}
	
	Let us start by recalling the results for the problem with known drift $\mu_i$ for $i \in \{0,1\}$ (see \cite[Sec. 6.1]{DixitPindyck1994}): 
	For these problems a threshold strategy is optimal, that is to stop whenever the profits exceed the threshold 
	\begin{displaymath}
		x^*_i := \frac{\beta_i}{\beta_i-1} (r-\mu_i) I,
	\end{displaymath} 
	where $\beta_i$ is the unique solution of 
	\begin{equation}
		\label{eq:BetaQuadraticEquation}
		\frac{1}{2} \sigma^2 \beta_i(\beta_i -1) + \mu_i \beta_i - r = 0
	\end{equation} 
	that satisfies $\beta_i >1$.
	Some easy computations (see Lemma~\ref{lem:beta_0_beta_1}) show that $\beta_0 > \beta_1$ and $x_0^\ast > x_1^\ast$.
	Since the belief process $\Pi^i$ is absorbing for $i \in \{0,1\}$, the problem for initial values $(x,i)$ reduces to the problem with known drift. 
	Hence, we immediately have that $b(0) = x_0^\ast> x_1 ^\ast = b(1)$.
	 	
	With these preparations, we can state the main result on the properties of the boundary function:
	\begin{theorem}
		\label{thm:PropertiesBoundaryFunction}
		The boundary function $b$ is decreasing, continuous, has the image $b([0,1]) = [x_1^*,x_0^*]$ and is bounded from below by the decreasing function $\underline{b}: [0,1] \to [x_1^*,x_0^*]$,
		\begin{equation}
			\label{eq:LowerBoundForBoundaryFunction}
			\pi \mapsto \underline{b}(\pi) := \frac{\beta_0 (1-\pi) + \beta_1 \pi}{ (\beta_0-1) \frac{1-\pi}{r-\mu_0} + (\beta_1-1) \frac{\pi}{r-\mu_1} } I.
		\end{equation}
	\end{theorem}
	
	In the rest of this subsection we prove this theorem in several steps: 
	First, we prove that $\underline{b}$ is indeed a lower bound of the boundary function. 
	Focusing on stopping times of the form \eqref{eq:Form_Stopping_Times} and utilizing the properties established in Section~\ref{sec:Transformation}, we then establish that the boundary function is decreasing and locally Lipschitz continuous on $(0,1)$. 
	Finally, we establish the continuity in $0$ and $1$.
		
	\begin{lemma}
		\label{lem:BoundaryFunctionLowerBound}
		The boundary function $b$ is bounded from below by $\underline{b}$ from \eqref{eq:LowerBoundForBoundaryFunction}. 
	\end{lemma}
	\begin{proof}
		Let us fix $\pi \in [0,1]$. 
		We prove that $b(\pi) \ge \underline{b}(\pi)$ by showing that any $x \ge b(\pi)$ satisfies 
		\begin{equation}
			\label{eq:DeltaStrategyDifferentiatedPayoff}
			x \left[ (1-\beta_0) \frac{1-\pi}{r-\mu_0} + (1-\beta_1) \frac{\pi}{r-\mu_1} \right] + \left[ \beta_0 (1-\pi) + \beta_1 \pi \right] I \le 0.
		\end{equation} 
		Rearranging this term for $x=b(\pi)$ immediately yields the desired claim.
		To prove \eqref{eq:DeltaStrategyDifferentiatedPayoff}, we define for every $\delta \ge 1$ the stopping time
		\begin{displaymath}
			\gamma^{x,\pi}_\delta := \inf\{t \geq 0 : X^{x,\pi}_t \ge x \delta \}.
		\end{displaymath} 
		For $i \in \{0,1\}$, we use \cite[p.622, Eq. 2.0.1]{BorodinSalminen2002} to obtain
		\begin{displaymath}
			\mathbb{E} \left[ e^{-r\gamma^{x,i}_\delta} \left( \frac{X^{x,i}_{\gamma^{x,i}_\delta}}{r-\mu_i} - I \right) \right] = \left(\frac{x}{x\delta}\right)^{\beta_i} \left( \frac{x\delta}{r-\mu_i}-I\right) = \delta^{1-\beta_i} \frac{x}{r-\mu_i} - \delta^{-\beta_i} I.
		\end{displaymath}
		Since the stopping time $\gamma^{x,\pi}_\delta$ can be decomposed as $\gamma^{x,\pi}_\delta = \mathbb{I}_{\{\theta^\pi=0\}} \gamma^{x,0}_\delta + \mathbb{I}_{\{\theta^\pi=1\}} \gamma^{x,1}_\delta$, we have
		\begin{align*}
			V(x,\pi) 
			&\ge \mathbb{E} \left[ e^{-r\gamma^{x,\pi}_\delta} \left( \frac{X^{x,\pi}_{\gamma^{x,\pi}_\delta}}{r-\mu_{\theta^\pi}} - I \right) \right] \\
			&=  (1-\pi) \, \mathbb{E} \left[ e^{-r\gamma^{x,0}_\delta} \left( \frac{X^{x,0}_{\gamma^{x,0}_\delta}}{r-\mu_0} - I \right) \right] + \pi \, \mathbb{E} \left[ e^{-r\gamma^{x,1}_\delta} \left( \frac{X^{x,1}_{\gamma^{x,1}_\delta}}{r-\mu_1} - I \right) \right] \\
			&= (1-\pi) \left[ \delta^{1-\beta_0} \frac{x}{r-\mu_0}- \delta^{-\beta_0} I \right] + \pi \left[ \delta^{1-\beta_1} \frac{x}{r-\mu_1}- \delta^{-\beta_1} I \right] \\
			&= \underbrace{x \left[ \delta^{1-\beta_0} \frac{1-\pi}{r-\mu_0} + \delta^{1-\beta_1} \frac{\pi}{r-\mu_1} \right] - \left[ \delta^{-\beta_0} (1-\pi) + \delta^{-\beta_1} \pi \right] I}_{=: h(x, \pi, \delta)} .
		\end{align*}
		The derivative of $h$ with respect to $\delta$ evaluated at $\delta=1$ reads
		\begin{displaymath}
			\frac{\partial}{\partial \delta} h(x, \pi, 1) = x \left[ (1-\beta_0) \frac{1-\pi}{r-\mu_0} + (1-\beta_1) \frac{\pi}{r-\mu_1} \right] + \left[ \beta_0 (1-\pi) + \beta_1 \pi \right] I.
		\end{displaymath}
		Hence, if we assume that \eqref{eq:DeltaStrategyDifferentiatedPayoff} does not hold, then we have $\frac{\partial}{\partial \delta} h(x, \pi, 1)>0$. 
		This implies that there exists an $\delta > 1$ such that $h(x, \pi, \delta) > h(x, \pi, 1) = g(x,\pi)$, i.e., the payoff using the strategy $\gamma^{x,\pi}_\delta$ is strictly greater than the payoff for stopping immediately. 
		Thus, we have $(x,\pi) \notin D$, which contradicts $x \ge b(\pi)$. 
		Therefore, \eqref{eq:DeltaStrategyDifferentiatedPayoff} holds for all $x \ge b(\pi)$, which means that $\underline{b}$ is indeed a lower bound of $b$.
	\end{proof}

	\begin{lemma}
		\label{lem:BoundaryFunctionDecreasing}
		The boundary function $b$ is decreasing.
	\end{lemma}	
	\begin{proof}
		We start by showing that $b(0) \ge b(\pi) \ge b(1)$ for all $\pi \in [0,1]$: 
		Since $\underline{b}$ is a lower bound for $b$ by Lemma~\ref{lem:BoundaryFunctionLowerBound} and it is decreasing by Lemma~\ref{lem:LowerBoundDecreasing}, we immediately have $b(\pi ) \ge \underline{b}(\pi) \ge \underline{b}(1) = x_1^\ast = b(1)$.
		Moreover, since $b(0) > b(1)$, we have $V(b(0),1) = g(b(0),1)$. 
		With this and Corollary~\ref{cor:UpperBoundValueFct}, we obtain
		\begin{align*}
			V(b(0),\pi) 
			\le (1-\pi) V(b(0),0) + \pi V(b(0),1) 
			= (1-\pi) g(b(0),0) + \pi g(b(0),1) 
			= g(b(0),\pi).
		\end{align*}
		So, $V(b(0),\pi) = g(b(0), \pi)$, which, in turn, implies $(b(0),\pi) \in D$ and hence $b(\pi) \le b(0)$. 
		
		With these preparations, it suffices to show that for all $\pi,\hat{\pi} \in [0,1]$ with $\pi < \hat{\pi}$ and $x \in [b(1),b(0)]$ satisfying $(x,\pi) \in D$, it holds that $(x,\hat{\pi}) \in D$:
		As $(x,\pi) \in D$, we have $V(x,\pi) = g(x,\pi)$. 
		Since $x \ge b(1)$, we have $V(x,1) = g(x,1)$.
		With this and Lemma~\ref{lem:EssentialIdentity}, we obtain
		\begin{align*}
			V(x,\hat{\pi}) 
			&= \mathbb{E} \left[ e^{-r\tau^{x,\hat{\pi}}_{\hat{\pi}}} \left( \frac{ X^{x,\hat{\pi}}_{\tau^{x,\hat{\pi}}_{\hat{\pi}}} }{ r-\mu_{\theta^{\hat{\pi}}}} -I \right) \right] \\
			&= (1-\hat{\pi}) \, \mathbb{E} \left[ e^{-r\tau^{x,0}_{\hat{\pi}}} \left( \frac{X^{x,0}_{\tau^{x,0}_{\hat{\pi}}}}{r-\mu_0} -I \right) \right] + \hat{\pi} \, \mathbb{E} \left[ e^{-r\tau^{x,1}_{\hat{\pi}}} \left( \frac{X^{x,1}_{\tau^{x,1}_{\hat{\pi}}}}{r-\mu_1} -I \right) \right] \\
			&= \frac{1-\hat{\pi}}{1-\pi} (1-\pi) \, \mathbb{E} \left[ e^{-r\tau^{x,0}_{\hat{\pi}}} \left( \frac{X^{x,0}_{\tau^{x,0}_{\hat{\pi}}}}{r-\mu_0} -I \right) \right] \\
			&\quad + \frac{1-\hat{\pi}}{1-\pi} \pi \, \mathbb{E} \left[ e^{-r\tau^{x,1}_{\hat{\pi}}} \left( \frac{X^{x,1}_{\tau^{x,1}_{\hat{\pi}}}}{r-\mu_1} -I \right) \right] + \left( \hat{\pi} - \frac{1-\hat{\pi}}{1-\pi} \pi \right) \, \mathbb{E} \left[ e^{-r\tau^{x,1}_{\hat{\pi}}} \left( \frac{X^{x,1}_{\tau^{x,1}_{\hat{\pi}}}}{r-\mu_1} -I \right) \right] \\
			&= \frac{1-\hat{\pi}}{1-\pi} \, \mathbb{E} \left[ e^{-r\tau^{x,\pi}_{\hat{\pi}}} \left( \frac{X^{x,\pi}_{\tau^{x,\pi}_{\hat{\pi}}}}{r-\mu_{\theta^\pi}} -I \right) \right] + \left( \hat{\pi} - \frac{1-\hat{\pi}}{1-\pi} \pi \right) \, \mathbb{E} \left[ e^{-r\tau^{x,1}_{\hat{\pi}}} \left( \frac{X^{x,1}_{\tau^{x,1}_{\hat{\pi}}}}{r-\mu_1} -I \right) \right] \\
			&\le \frac{1-\hat{\pi}}{1-\pi} \, V(x,\pi) + \left( \hat{\pi} - \frac{1-\hat{\pi}}{1-\pi} \pi \right) \, V(x,1) 
			= \frac{1-\hat{\pi}}{1-\pi} \, g(x,\pi) + \left( \hat{\pi} - \frac{1-\hat{\pi}}{1-\pi} \pi \right) \, g(x,1) \\
			&= g(x,\hat{\pi}).
		\end{align*}
		Hence, we have $V(x,\hat{\pi}) = g(x,\hat{\pi})$, which, in turn, implies $(x, \hat{\pi}) \in D$.
	\end{proof}
	
	\begin{lemma}
		\label{lem:BoundaryFunctionLocallyLipschitz}
		The boundary function $b$ is locally Lipschitz continuous on $(0,1)$.
	\end{lemma}
	\begin{proof}
		Let $\pi, \hat{\pi} \in (0,1)$ with $\pi < \hat{\pi}$. 
		We consider
		\begin{displaymath}
			x := \frac{b(\hat{\pi})}{ 1 - \frac{\hat{\pi}-\pi}{\hat{\pi} (1-\pi)} \left( 1 - \frac{b(\hat{\pi})}{b(0)} \right)} > b(\hat{\pi})
		\end{displaymath}
		and notice that
		\begin{equation}
			\label{eq:UpperBoundForBoundaryFunctionEquation}
			(1-\pi) \left( \frac{x}{b(\hat{\pi})} -1 \right) - \frac{x}{b(\hat{\pi})} \left( 1-  \frac{\pi}{\hat{\pi}} \right) \left( 1- \frac{b(\hat{\pi}) }{b(0)}\right) = 0.
		\end{equation}
		We now show that $(x,\pi) \in D$: 
		For this we obtain using \eqref{eq:xShift} and Lemma~\ref{lem:EssentialIdentity} that 
		\begin{align*}
			V(x,\pi) 
			&= \mathbb{E} \left[ e^{-r\tau^{x,\pi}_\pi} \left( \frac{X^{x,\pi}_{\tau^{x,\pi}_\pi}}{r-\mu_{\theta^\pi}}-I\right) \right] \\
			&= \frac{x}{b(\hat{\pi})} \mathbb{E} \left[ e^{-r\tau^{x,\pi}_\pi} \left( \frac{X^{b(\hat{\pi}),\pi}_{\tau^{x,\pi}_\pi}}{r-\mu_{\theta^\pi}}-I\right) \right] + \left( \frac{x}{b(\hat{\pi})} -1 \right) \mathbb{E} \left[ e^{-r\tau^{x,\pi}_\pi} I \right] \\
			&= \frac{x}{b(\hat{\pi})} \left( (1-\pi) \mathbb{E} \left[ e^{-r\tau^{x,0}_\pi} \left( \frac{X^{b(\hat{\pi}),0}_{\tau^{x,0}_\pi}}{r-\mu_0}-I\right) \right] + \pi \mathbb{E} \left[ e^{-r\tau^{x,1}_\pi} \left( \frac{X^{b(\hat{\pi}),1}_{\tau^{x,1}_\pi}}{r-\mu_1}-I\right) \right] \right) \\
			&\quad + \left( \frac{x}{b(\hat{\pi})} -1 \right) \mathbb{E} \left[ e^{-r\tau^{x,\pi}_\pi} I \right] \\
			&= \frac{x}{b(\hat{\pi})} \Bigg( \frac{\pi}{\hat{\pi}}\hat{\pi} \mathbb{E} \left[ e^{-r\tau^{x,1}_\pi} \left( \frac{X^{b(\hat{\pi}),1}_{\tau^{x,1}_\pi}}{r-\mu_1}-I\right) \right]  + \frac{\pi}{\hat{\pi}}(1-\hat{\pi}) \mathbb{E} \left[ e^{-r\tau^{x,0}_\pi} \left( \frac{X^{b(\hat{\pi}),0}_{\tau^{x,0}_\pi}}{r-\mu_0}-I\right) \right] \\
			&\quad + \left( 1-\pi - \frac{\pi}{\hat{\pi}}(1-\hat{\pi}) \right) \mathbb{E} \left[ e^{-r\tau^{x,0}_\pi} \left( \frac{X^{b(\hat{\pi}),0}_{\tau^{x,0}_\pi}}{r-\mu_0}-I\right) \right] \Bigg) + \left( \frac{x}{b(\hat{\pi})} -1 \right) \mathbb{E} \left[ e^{-r\tau^{x,\pi}_\pi} I \right] \\
			&= \frac{x}{b(\hat{\pi})} \frac{\pi}{\hat{\pi}}  \mathbb{E} \left[ e^{-r\tau^{x,\hat{\pi}}_\pi} \left( \frac{X^{b(\hat{\pi}),\hat{\pi}}_{\tau^{x,\hat{\pi}}_\pi}}{r-\mu_{\theta^{\hat{\pi}}}}-I\right) \right] \\
			&\quad + \frac{x}{b(\hat{\pi})} \left( 1- \frac{\pi}{\hat{\pi}} \right) \mathbb{E} \left[ e^{-r\tau^{x,0}_\pi} \left( \frac{X^{b(\hat{\pi}),0}_{\tau^{x,0}_\pi}}{r-\mu_0}-I\right) \right]  + \left( \frac{x}{b(\hat{\pi})} -1 \right) \mathbb{E} \left[ e^{-r\tau^{x,\pi}_\pi} I \right].
		\end{align*}
		Using \eqref{eq:TauDecomposition}~and~\eqref{eq:xShift} for the first equality and \eqref{eq:UpperBoundForBoundaryFunctionEquation} for the second, we observe that
		\begin{align*}
			V(x,\pi) 
			&= \frac{x}{b(\hat{\pi})} \frac{\pi}{\hat{\pi}}  \mathbb{E} \left[ e^{-r\tau^{x,\hat{\pi}}_\pi} \left( \frac{X^{b(\hat{\pi}),\hat{\pi}}_{\tau^{x,\hat{\pi}}_\pi}}{r-\mu_{\theta^{\hat{\pi}}}}-I\right) \right] \\
			&\quad + \frac{x}{b(\hat{\pi})} \left( 1- \frac{\pi}{\hat{\pi}} \right) \left( \frac{b(\hat{\pi})}{b(0)} \mathbb{E} \left[ e^{-r\tau^{x,0}_\pi} \left( \frac{X^{b(0),0}_{\tau^{x,0}_\pi}}{r-\mu_0}-I\right) \right] + \left( \frac{b(\hat{\pi}) }{b(0)} -1\right) \mathbb{E} \left[ e^{-r\tau^{x,0}_\pi} I \right] \right) \\
			&\quad + \left( \frac{x}{b(\hat{\pi})} -1 \right) \left( (1-\pi) \mathbb{E} \left[ e^{-r\tau^{x,0}_\pi} I \right] + \pi \mathbb{E} \left[ e^{-r\tau^{x,1}_\pi} I \right] \right) \\
			&= \frac{x}{b(\hat{\pi})} \frac{\pi}{\hat{\pi}}  \mathbb{E} \left[ e^{-r\tau^{x,\hat{\pi}}_\pi} \left( \frac{X^{b(\hat{\pi}),\hat{\pi}}_{\tau^{x,\hat{\pi}}_\pi}}{r-\mu_{\theta^{\hat{\pi}}}}-I\right) \right] + \frac{x}{b(0)} \left( 1- \frac{\pi}{\hat{\pi}} \right) \mathbb{E} \left[ e^{-r\tau^{x,0}_\pi} \left( \frac{X^{b(0),0}_{\tau^{x,0}_\pi}}{r-\mu_0}-I\right) \right] \\
			&\quad + \left( \frac{x}{b(\hat{\pi})} -1 \right) \pi \mathbb{E} \left[ e^{-r\tau^{x,1}_\pi} I \right].
		\end{align*}
		Utilizing that $x > b(\hat{\pi})$, we get
		\begin{align*}
			V(x,\pi) 
			&\le \frac{x}{b(\hat{\pi})} \frac{\pi}{\hat{\pi}} V(b(\hat{\pi}),\hat{\pi}) + \frac{x}{b(0)} \left( 1- \frac{\pi}{\hat{\pi}} \right) V(b(0),0) + \left( \frac{x}{b(\hat{\pi})} -1 \right) \pi I \\
			&= \frac{x}{b(\hat{\pi})} \frac{\pi}{\hat{\pi}} g(b(\hat{\pi}),\hat{\pi}) + \frac{x}{b(0)} \left( 1- \frac{\pi}{\hat{\pi}} \right) g(b(0),0) + \left( \frac{x}{b(\hat{\pi})} -1 \right) \pi I \\
			&= x \left( \frac{1-\pi}{r-\mu_0} + \frac{\pi}{r-\mu_1} \right) - I \left( 1 + (1-\pi) \left(\frac{x}{b(\hat{\pi})}-1 \right) - \frac{x}{b(\hat{\pi})} \left(1- \frac{\pi}{\hat{\pi}} \right) \left( 1 -\frac{b(\hat{\pi})}{b(0)} \right) \right).
		\end{align*}
		By \eqref{eq:UpperBoundForBoundaryFunctionEquation} this means that $V(x,\pi) = g(x,\pi)$, which implies $(x, \pi) \in D$. 
		Hence, by definition of $b$, we have $b(\pi) \le x$, which, in turn, yields
		\begin{displaymath}
			b(\hat{\pi}) \ge b(\pi) \left[ 1 - \frac{\hat{\pi}-\pi}{\hat{\pi} (1-\pi)} \left( 1 - \frac{b(\hat{\pi})}{b(0)} \right) \right].
		\end{displaymath}
		Utilizing this inequality and $b(0) \ge b(\pi) \ge b(\hat{\pi}) \ge b(1)$, we obtain
		\begin{displaymath}
			0 
			\le b(\pi) - b(\hat{\pi}) 
			\le b(\pi) \frac{\hat{\pi}-\pi}{\hat{\pi} (1-\pi)} \left( 1 - \frac{b(\hat{\pi})}{b(0)} \right)
			\le |\hat{\pi}-\pi| \frac{b(0)-b(1)}{\hat{\pi} (1-\pi)}.
		\end{displaymath}
		Hence, $b$ is locally Lipschitz continuous on $(0,1)$.
	\end{proof}
	\begin{lemma}
		The boundary function $b$ is right-continuous on $0$.
	\end{lemma}
	\begin{proof}
		By Lemma \ref{lem:BoundaryFunctionDecreasing}, the boundary function $b$ is decreasing. 
		Hence, $b(0+) := \lim_{\pi \downarrow 0} b(\pi)$ exists and $b(0+) \le b(0)$ holds. 
		Since $\underline{b}$ from \eqref{eq:LowerBoundForBoundaryFunction} is continuous and is a lower bound for $b$, we have $b(0+) \ge \lim_{\pi \downarrow 0} \underline{b}(\pi) = b(0)$.
	\end{proof}
	
	\begin{lemma}
		\label{lem:BoundaryFunctionLeftContinuousOn1}
		The boundary function $b$ is left-continuous at $1$.
	\end{lemma}
	\begin{proof} 
		By Lemma~\ref{lem:BoundaryFunctionDecreasing} the boundary function $b$ is decreasing. 
		Hence, $b(1-) := \lim_{\pi \uparrow 1} b(\pi)$ exists and satisfies $b(1-) \ge b(1)$. 
		It remains to prove the reverse inequality: 
		By way of contradiction, we assume that there is an $x \in (0,\infty)$ such that $b(1)< x < b(1-)$. 
		We set $\delta := x / b(1)$ and note that $\delta>1$. 
		Let $\pi \in (0,1)$ be arbitrary. 
		Then we define the stopping time 
		\begin{displaymath}
			\gamma^{b(1),\pi}_\delta = \inf\{t \ge 0 : X^{b(1),\pi}_t \ge \delta b(1) \},
		\end{displaymath} 
		which admits the decomposition $\gamma^{b(1),\pi}_\delta = \mathbb{I}_{\{\theta^\pi = 0\}} \gamma^{b(1),0}_\delta + \mathbb{I}_{\{\theta^\pi = 1\}} \gamma^{b(1),1}_\delta$. 
		Since $\Pi^\pi_t \in (0,1)$ for all $t \ge 0$~a.s., we have that $b(\Pi_t^\pi) \ge b(1-) \ge \delta b(1)$ for all $t \ge 0$~a.s. 
		Hence, $\tau^{b(1),\pi}_\pi \ge  \gamma^{b(1),\pi}_\delta$~a.s. 
		This fact allows us to rewrite the  optimal stopping time for the stopping problem $V(x,\pi)$ as 
		\begin{displaymath}
			\tau_\ast (b(1), \pi) = \gamma_\delta^{b(1), \pi} + \tau_\ast \left(x,\Pi^\pi_{\gamma_\delta^{b(1), \pi}}\right) \circ \theta_{\gamma_\delta^{b(1),\pi}}.
		\end{displaymath}
	 	This and the strong Markov property (e.g., \cite[p.322, Theorem 4.20]{KaratzasShreve1998}) yields
		\begin{align*}
			g(b(1),\pi) &\le V(b(1), \pi) 
			= \mathbb{E} \left[ \mathbb{E} \left[ e^{-r \tau_\ast (b(1), \pi)} g\left( X_{\tau_\ast (b(1), \pi)}^{b(1),\pi}, \Pi^\pi_{\tau_\ast (b(1), \pi)} \right) \middle| \mathcal{F}_{\gamma_\delta^{b(1),\pi}}^{X^{b(1), \pi}, \Pi^\pi} \right] \right] \\
			&= \mathbb{E} \left[ e^{-r \gamma_\delta^{b(1), \pi}} 
			\mathbb{E} \left[ e^{-r \tau_\ast \left(x,\Pi^\pi_{\gamma_\delta^{b(1), \pi}}\right)}
			g \left( X_{\tau_\ast \left(x,\Pi^\pi_{\gamma_\delta^{b(1), \pi}}\right)}^{x, \Pi^\pi_{\gamma_\delta^{b(1), \pi}}}, 
			\Pi^{\Pi^\pi_{\gamma_\delta^{b(1), \pi}}}_{\tau_\ast \left(x,\Pi^\pi_{\gamma_\delta^{b(1), \pi}}\right)} 
			\right) 
			\right] \right] \\
			&= \mathbb{E} \left[  e^{-r \gamma_\delta^{b(1), \pi}} V\left(x, \Pi^\pi_{\gamma_\delta^{b(1), \pi}} \right) \right].
		\end{align*}
		By Lemma~\ref{lem:ValueFunctionIncreasing}, the value function $V$ is increasing in its second argument. 
		With this and \cite[p.622, Eq. 2.0.1]{BorodinSalminen2002}, we obtain 
		\begin{align*}
			g(b(1),\pi) 
			&\le V(x,1)  \mathbb{E} \left[  e^{-r \gamma_\delta^{b(1), \pi}} \right] 
			= V(x,1) \left( (1-\pi) \mathbb{E} \left[ e^{-r\gamma^{b(1),0}_\delta}  \right] + \pi  \mathbb{E} \left[ e^{-r\gamma^{b(1),1}_\delta}  \right] \right) \\
			&= V(x,1) \left( (1-\pi) \delta^{-\beta_0} + \pi \delta^{-\beta_1} \right).
		\end{align*}
		Since $x \ge b(1)$, we have $V(x,1)=g(x,1)$. 
		Using this and the fact that $\beta_0 > \beta_1$, we obtain 
		\begin{equation}
			\label{eq:BoundaryFunctionLeftContinuousIn1Contradiction}
				 \delta^{-\beta_1} g(x,1) - g(b(1),\pi) >0.
		\end{equation}
		Recall that $b(1) = \frac{\beta_1}{\beta_1-1} (r-\mu_1) I $ and $x= \delta b(1)$. 
		Hence, we obtain 
		\begin{displaymath}
			\delta^{-\beta_1} g(x,1) - g(b(1),1)
			= I \left( \frac{\beta_1}{\beta_1 -1} ( \delta^{1-\beta_1}-1) - 	\delta^{-\beta_1} + 1 \right).
		\end{displaymath}
		We observe that $h(\hat{\delta}):= \frac{\beta_1}{\beta_1 -1} ( \hat{\delta}^{1-\beta_1}-1) - \hat{\delta}^{-\beta_1} + 1$ satisfies $h(1)=0$ and  $h'(\hat{\delta}) = - \beta_1 \hat{\delta}^{-\beta_1} (1-\hat{\delta}^{-1})$.
		Since $h'(\hat{\delta})<0$ for all $\hat{\delta}>1$, we have $h(\hat{\delta})<0$ for all $\hat{\delta}>1$. 
		Hence, $\delta^{-\beta_1} g(x,1) - g(b(1),1)<0$. 
		By continuity of $g$, there is a $\pi \in (0,1)$ such that $\delta^{-\beta_1} g(x,1) - g(b(1),\pi)<0$, which contradicts \eqref{eq:BoundaryFunctionLeftContinuousIn1Contradiction}.
	\end{proof}
	
	\subsection{Nonlinear Integral Equation for the Boundary Function}
	\label{sec:Integral_boundary}
	Our goal for the remainder of this section is to derive an integral equation, which is the basis of our numerical iteration scheme. 
	More precisely, we prove that the boundary function is the unique solution of an integral equation in the class of all functions that satisfy the conditions of Theorem~\ref{thm:PropertiesBoundaryFunction}. A similar approach has been successfully used in \cite{ChristensenCrocceMordeckiSalminen2019, ChristensenSalminen2018, DammannFerrari2022,DetempleKitapbayev2020} among others. 
	\begin{theorem}
		\label{thm:NonlinerIntegralEquation}
		The boundary function $b$ is the unique function $a \in \mathcal{M}$ with
		\begin{displaymath}
			\mathcal{M} := \{ a:[0,1] \to [x^*_1,x^*_0] \text{ continuous, decreasing with } \underline{b}(\pi) \le a(\pi) \text{ for all }\pi \in [0,1] \text{ and } a(1)=x^*_1 \}
		\end{displaymath} 
		such that
		\begin{equation}
			\label{eq:NonlinearIntegralEquation}
			\mathbb{E} \left[ \int_0^\infty e^{-rt} (X^{a(\pi), \pi}_t -rI) \mathbb{I}_{\{X^{a(\pi), \pi}_t \le a(\Pi^\pi_t)\}} \, dt \right] = 0
		\end{equation}
		for all $\pi \in [0,1]$.
	\end{theorem}
	
	In the remainder of this subsection we prove this theorem relying on the free boundary approach to optimal stopping. First we prove that the value function is continuously differentiable, then we establish an integral representation of the value function. This allows us to prove the result.

	As usual, the infinitesimal generator $\mathcal{L}$ of the process $(X^{\hat{x},\hat{\pi}},\Pi^{\hat{\pi}})$  from \eqref{eq:XPiDynamics}  is given by
	\begin{align*}
		\mathcal{L} u(x,\pi)
		&=  \left(\mu_0 + \pi (\mu_1-\mu_0) \right) x \frac{\partial}{\partial x} u(x,\pi)
		+ \frac{1}{2}\sigma^2 x^2 \frac{\partial^2}{\partial x^2} u(x,\pi)
		+ x (\mu_1-\mu_0) \pi (1-\pi) \frac{\partial^2}{\partial x \partial \pi} u(x,\pi) \\
		&\quad + \frac{1}{2} \left(\frac{\mu_1-\mu_0}{\sigma}\right)^2 \pi^2 (1-\pi)^2 \frac{\partial^2}{\partial \pi^2} u(x,\pi)
	\end{align*} 
	defined for all $u \in C^2((0,\infty) \times (0,1))$ and $(x, \pi) \in (0, \infty) \times (0,1)$.
	Since the first entry time into the set $D$ is optimal, we obtain by 
	\cite[Ch. III, Sec. 7.1]{PeskirShiryaev2006} that the continuous value function $V$ is actually $C^2((0,\infty) \times (0,1) \setminus \partial D)$ and solves the free boundary problem
	\begin{align}
		\label{eq:FreeBoundaryProblem}
		\begin{split}
		(\mathcal{L}-r)V &= 0 \quad \text{on } (0,\infty) \times (0,1) \setminus D \\
		V &= g \quad \text{on } D.
		\end{split}
	\end{align}
	
	Moreover, we obtain that in our setting the value function $V$ also satisfies the smooth fit conditions:
	\begin{lemma}
		\label{lem:ValueFunctionC^1}
		The value function $V$ is $C^1((0,\infty) \times (0,1))$ and satisfies
		\begin{displaymath}
			\frac{\partial}{\partial x} V(b(\pi),\pi) = \frac{\partial}{\partial x} g(b(\pi),\pi) \quad \text{and} \quad	\frac{\partial}{\partial \pi} V(b(\pi),\pi) = \frac{\partial}{\partial \pi} g(b(\pi),\pi) \quad \text{for all } \pi \in (0,1).
		\end{displaymath}
	\end{lemma}
	\begin{proof}
		It suffices to prove that $V$ is continuously differentiable on the boundary $\partial D$. 
		Fix $\pi \in (0,1)$. 
		Then for $\varepsilon >0$ we obtain from \eqref{eq:xShift} that
		\begin{align*}
			V(b(\pi)-\varepsilon, \pi) 
			&= \frac{b(\pi)-\varepsilon}{b(\pi)} \mathbb{E} \left[ e^{-r \tau^{b(\pi)-\varepsilon,\pi}_\pi} g(X^{b(\pi),\pi}_{\tau^{b(\pi)-\varepsilon,\pi}_\pi}, \Pi^{\pi}_{\tau^{b(\pi)-\varepsilon,\pi}_\pi}) \right] + \left( \frac{b(\pi)-\varepsilon}{b(\pi)} - 1 \right) \mathbb{E} \left[ e^{-r\tau^{b(\pi)-\varepsilon,\pi}_\pi} \right] I \\
			&\le \frac{b(\pi)-\varepsilon}{b(\pi)} V(b(\pi),\pi) -\varepsilon \frac{I}{b(\pi)} \mathbb{E} \left[ e^{-r\tau^{b(\pi)-\varepsilon,\pi}_\pi} \right] \\
			&= \frac{b(\pi)-\varepsilon}{b(\pi)} g(b(\pi),\pi) -\varepsilon \frac{I}{b(\pi)} \mathbb{E} \left[ e^{-r\tau^{b(\pi)-\varepsilon,\pi}_\pi} \right] \\
			&= g(b(\pi)-\varepsilon,\pi) + \varepsilon \frac{I}{b(\pi)} \left(1-  \mathbb{E} \left[ e^{-r\tau^{b(\pi)-\varepsilon,\pi}_\pi} \right] \right).
		\end{align*}
		With this we obtain
		\begin{align*}
			g(b(\pi),\pi) - g(b(\pi)-\varepsilon, \pi) 
			&\ge V(b(\pi),\pi) - V(b(\pi)-\varepsilon, \pi) \\
			&\ge g(b(\pi),\pi) - g(b(\pi)-\varepsilon,\pi) - \varepsilon \frac{I}{b(\pi)} \left(1- \mathbb{E} \left[ e^{-r\tau^{b(\pi)-\varepsilon,\pi}_\pi} \right] \right).
		\end{align*}
		Dividing by $\varepsilon$ and taking the limit $\varepsilon \rightarrow 0$, we obtain using that $\mathbb{E} \left[ e^{-r\tau^{b(\pi)-\varepsilon,\pi}_\pi} \right] \to 1$ for $\varepsilon \to 0$ 
		\begin{align*}
			\lim_{\varepsilon \rightarrow 0} \frac{V(b(\pi), \pi) - V(b(\pi) -\varepsilon, \pi)}{\varepsilon} = \lim_{\varepsilon \rightarrow 0} \frac{g(b(\pi), \pi) - g(b(\pi) -\varepsilon, \pi)}{\varepsilon}.
		\end{align*}
		Similarly, we obtain for every $\varepsilon >0$ from Lemma~\ref{lem:EssentialIdentity} that
		\begin{align*}
			V(b(\pi), \pi-\varepsilon) 
			&= (1-\pi+\varepsilon) \mathbb{E} \left[ e^{-r\tau^{b(\pi),0}_{\pi-\varepsilon}} g(X^{b(\pi),0}_{\tau^{b(\pi),0}_{\pi - \varepsilon}}, 0) \right] + (\pi - \varepsilon) \mathbb{E} \left[ e^{-r\tau^{b(\pi),1}_{\pi-\varepsilon}} g(X^{b(\pi),1}_{\tau^{b(\pi),1}_{\pi - \varepsilon}}, 1) \right] \\
			&= \frac{\pi-\varepsilon}{\pi} (1-\pi) \mathbb{E} \left[ e^{-r\tau^{b(\pi),0}_{\pi-\varepsilon}} g(X^{b(\pi),0}_{\tau^{b(\pi),0}_{\pi - \varepsilon}}, 0) \right] 
			+ \frac{\pi-\varepsilon}{\pi} \pi \mathbb{E} \left[ e^{-r\tau^{b(\pi),1}_{\pi-\varepsilon}} g(X^{b(\pi),1}_{\tau^{b(\pi),1}_{\pi - \varepsilon}}, 1) \right] \\
			&\quad + \left( 1 - \pi + \varepsilon - \frac{\pi-\varepsilon}{\pi} (1-\pi) \right) \mathbb{E} \left[ e^{-r\tau^{b(\pi),0}_{\pi-\varepsilon}} g(X^{b(\pi),0}_{\tau^{b(\pi),0}_{\pi - \varepsilon}}, 0) \right] \\
			&\le \frac{\pi-\varepsilon}{\pi} V(b(\pi),\pi) + \frac{\varepsilon}{\pi} \mathbb{E} \left[ e^{-r\tau^{b(\pi),0}_{\pi-\varepsilon}} g(X^{b(\pi),0}_{\tau^{b(\pi),0}_{\pi - \varepsilon}}, 0) \right] \\
			&= \frac{\pi-\varepsilon}{\pi} g(b(\pi),\pi) + \frac{\varepsilon}{\pi} \mathbb{E} \left[ e^{-r\tau^{b(\pi),0}_{\pi-\varepsilon}} g(X^{b(\pi),0}_{\tau^{b(\pi),0}_{\pi - \varepsilon}}, 0) \right] \\
			&= g(b(\pi), \pi-\varepsilon) - \frac{\varepsilon}{\pi} \left( g(b(\pi),0) - \mathbb{E} \left[ e^{-r\tau^{b(\pi),0}_{\pi-\varepsilon}} g(X^{b(\pi),0}_{\tau^{b(\pi),0}_{\pi - \varepsilon}}, 0) \right] \right).
		\end{align*}
		With this we obtain
		\begin{align*}
			g(b(\pi),\pi) - g(b(\pi), \pi-\varepsilon) 
			&\ge V(b(\pi),\pi) - V(b(\pi), \pi-\varepsilon)  \\
			&\ge g(b(\pi),\pi) - g(b(\pi), \pi-\varepsilon) + \frac{\varepsilon}{\pi} \left( g(b(\pi),0) - \mathbb{E} \left[ e^{-r\tau^{b(\pi),0}_{\pi-\varepsilon}} g(X^{b(\pi),0}_{\tau^{b(\pi),0}_{\pi - \varepsilon}}, 0) \right] \right)
		\end{align*}
		Dividing by $\varepsilon$, taking the limit and using $\mathbb{E} \left[ e^{-r\tau^{b(\pi),0}_{\pi - \varepsilon}} g(X^{b(\pi),0}_{\tau^{b(\pi),0}_{\pi - \varepsilon}}, 0) \right] \to g(b(\pi),0)$ for $\varepsilon \to 0$, we have
		\begin{displaymath}
			\lim_{\varepsilon \rightarrow 0} \frac{V(b(\pi), \pi) - V(b(\pi) , \pi-\varepsilon)}{\varepsilon} = \lim_{\varepsilon \rightarrow 0} \frac{g(b(\pi), \pi) - g(b(\pi), \pi -\varepsilon)}{\varepsilon}.
		\end{displaymath} 
		Since $g$ is smooth, the value function is indeed $C^1((0,\infty) \times (0,1))$.
	\end{proof}
	
	\begin{lemma}
		\label{lem:ValueFunctionRepresentation}
		The value function admits the representation
		\begin{displaymath}
			V(x,\pi) = \mathbb{E} \left[ \int_0^\infty e^{-rt} (X^{x,\pi}_t - rI) \mathbb{I}_{\{X^{x,\pi}_t \ge b(\Pi^\pi_t)\}} \, dt \right].
		\end{displaymath}
	\end{lemma}
	\begin{proof}	
		By Lemma \ref{lem:BoundaryFunctionLocallyLipschitz}, the boundary $\partial D$ is a Lipschitz surface.
		By Lemma~\ref{lem:ValueFunctionC^1}, the value function is $C^1((0,\infty) \times (0,1))$. 
		For $\pi \in \{0,1\}$, we moreover get from \cite[Sec. 6.1]{DixitPindyck1994} that $V(x,\pi) \rightarrow 0$ when $x \to 0$. 
		Together with Corollary~\ref{cor:UpperBoundValueFct}, we then obtain for any $\pi \in [0,1]$ that
		\begin{displaymath}
			0 \le \lim_{x \to 0} V(x,\pi) \le \lim_{x \to 0} ((1-\pi) V(x,0) + \pi V(x,1)) = 0,
		\end{displaymath} 
		which implies that $V$ is $C([0,\infty) \times [0,1])$. 
		Since $V$ solves \eqref{eq:FreeBoundaryProblem}, we finally have that the second derivatives of $V$ are locally bounded near $\partial D$.  
		Hence, we are in the setting of \cite[Theorem D.1]{Oksendal2003} and obtain that there is a sequence $\left(u_k\right)_{k \in \mathbb{N}}$ of $C^2((0,\infty) \times (0,1)) \cap C([0,\infty) \times [0,1])$-functions approximating $V$ such that $\left(u_k\right)_{k \in \mathbb{N}}$ converges uniformly to $V$ on compact subsets of $[0,\infty) \times [0,1]$, $((\mathcal{L}-r)u_k)_{k \in \mathbb{N}}$ converges uniformly to $(\mathcal{L}-r)V$ on compact subsets of $(0,\infty) \times (0,1) \setminus \partial D$ and $((\mathcal{L}-r)u_k)_{k \in \mathbb{N}}$ is locally bounded on $(0,\infty) \times (0,1)$. 
		Now, we take an increasing sequence of compact subsets $\left(K_n\right)_{n \in \mathbb{N}}$ of $(0,\infty) \times (0,1)$ such that, for all $(x,\pi) \in K_n$, we have $\gamma_n^{x,\pi} := \inf \{t \ge 0: (X^{x,\pi}_t, \Pi^\pi_t) \notin K_n\} \to \infty$ a.s.\ for $n \to \infty$. 
		For all $k,n \in \mathbb{N}$, we have that $\gamma_n^{x,\pi}$ has finite expectation and thus Dynkin's formula (see e.g. \cite[Theorem 7.4.1]{Oksendal2003}) yields
		\begin{equation}
			\label{eq:DynkinFormulaAppliedOnu_k}
			u_k (x,\pi) = \mathbb{E}\left[ e^{-r\gamma^{x,\pi}_n} u_k(X^{x,\pi}_{\gamma^{x,\pi}_n}, \Pi^\pi_{\gamma^{x,\pi}_n}) \right] - \mathbb{E}\left[ \int_0^{\gamma^{x,\pi}_n} e^{-rt} \left( \mathcal{L}-r \right) u_k(X^{x,\pi}_t,\Pi^\pi_t) \, dt \right].
		\end{equation} 
		For all $(x, \pi) \in (0,\infty) \times (0,1)$ we have $(\mathcal{L}-r)g(x,\pi)=rI - x$, hence, we obtain
		\begin{equation}
			\label{eq:GeneratorValueFunction}
			(\mathcal{L}-r)V(x,\pi) = \mathbb{I}_{D} (x,\pi) (\mathcal{L}-r)g(x,\pi)  = \mathbb{I}_D(x,\pi) (rI -x),
		\end{equation}
		where we defined $(\mathcal{L}-r)V(x,\pi) := (\mathcal{L}-r)g(x,\pi)$ on $(0,\infty) \times (0,1) \cap \partial D$. 	
		Since every $(x,\pi) \in D$ satisfies $x \ge b(1) \ge rI$ by Lemma~\ref{lem:beta_0_beta_1}, we have		
		$\left| \left( r-\mathcal{L} \right) V(x,\pi) \right| \le x$ for all $(x, \pi) \in (0,\infty) \times (0,1)$. This implies
		\begin{equation}
			\label{eq:IntegrabilityCondition}
			\mathbb{E}\left[ \int_0^\infty e^{-rt} \left| \left( r-\mathcal{L} \right) V(X^{x,\pi}_t,\Pi^\pi_t) \right| \, dt \right]
			\le \mathbb{E}\left[ \int_0^\infty e^{-rt} X^{x,\pi}_t \, dt \right] = \mathbb{E}\left[  \frac{x}{r-\mu_{\theta^\pi}}  \right] \le \frac{x}{r-\mu_1} < \infty.
		\end{equation}
		For fixed $n \in \mathbb{N}$, we remember that $K_n$ is compact, that $((\mathcal{L}-r)u_k)_{k \in \mathbb{N}}$ converges uniformly to $(\mathcal{L}-r)V$ on compact subsets of $(0,\infty) \times (0,1) \setminus \partial D$ and that $((\mathcal{L}-r)u_k)_{k \in \mathbb{N}}$ is locally bounded on $\partial D$.
		Moreover, the integral of any bounded function restricted to $K_n \cap \partial D$ is zero, since $\mathbb{E} \left[ \int_0^\infty \mathbb{I}_{\{ (X^{x,\pi}_t,\Pi^\pi_t) \in \partial D \}} \, dt \right] = 0$, which follows since $b$ is a continuous decreasing function and $X^{x,\pi}_t$ has a continuous density. 
		Hence, by application of dominated convergence for $k \to \infty$ in \eqref{eq:DynkinFormulaAppliedOnu_k}, we obtain
		\begin{displaymath}
			V(x,\pi) = \mathbb{E}\left[ e^{-r\gamma^{x,\pi}_n} V(X^{x,\pi}_{\gamma^{x,\pi}_n}, \Pi^\pi_{\gamma^{x,\pi}_n}) \right] - \mathbb{E}\left[ \int_0^{\gamma^{x,\pi}_n} e^{-rt} \left( \mathcal{L}-r \right) V(X^{x,\pi}_t,\Pi^\pi_t)  \, dt \right].
		\end{displaymath}
		Since $V$ is increasing in both arguments by Lemma~\ref{lem:ValueFunctionIncreasing} and $\mu_1 <r$, we obtain for all $(x,\pi) \in (0,\infty) \times [0,1]$ that
		\begin{align*}
			\lim_{t \to \infty} \mathbb{E} \left[ e^{-rt} V(X^{x,\pi}_t, \Pi^{\pi}_t) \right] 
			&\le \lim_{t \to \infty} e^{-rt} \mathbb{E} \left[ V(X^{x,\pi}_t + x_1^\ast, 1) \right]
			= \lim_{t \to \infty} e^{-rt} \mathbb{E} \left[ g(X^{x,\pi}_t + x_1^\ast, 1) 	\right] \\
			&\le \tfrac{1}{r- \mu_1} \lim_{t \to \infty} e^{-rt}  (x_1^* + \mathbb{E} 	\left[X^{x,\pi}_t\right])
			\le \tfrac{1}{r- \mu_1} \lim_{t \to \infty} \left[ x_1^\ast e^{-rt} + x e^{(\mu_1-r - \frac{1}{2} \sigma^2)t} \right]\\
			&= 0.
		\end{align*} 
		Again, from dominated convergence for $n \to \infty$, we get
		\begin{displaymath}
			V(x,\pi) = \mathbb{E}\left[ \int_0^\infty e^{-rt} \left( r - \mathcal{L} \right) V(X^{x,\pi}_t,\Pi^\pi_t) \, dt \right].
		\end{displaymath} 
		Eventually, the claim follows with \eqref{eq:GeneratorValueFunction}.
	\end{proof}
	
	\begin{proof}[Proof of Theorem~\ref{thm:NonlinerIntegralEquation}]
		As in \eqref{eq:Derivation_ValueFunctionWithoutIntegral}, we obtain for every  $(x,\pi) \in (0,\infty) \times [0,1]$ that
		\begin{displaymath}
			g(x,\pi) =\mathbb{E}\left[ \int_0^\infty e^{-rt} ( X^{x,\pi}_t - rI ) \, dt \right].
		\end{displaymath}
		Then, by Lemma~\ref{lem:ValueFunctionRepresentation} for $(b(\pi),\pi)$, we get
		\begin{displaymath}
			0 = V(b(\pi),\pi) - g(b(\pi),\pi) = \mathbb{E} \left[ \int_0^\infty e^{-rt} (X^{b(\pi),\pi}_t - rI) \mathbb{I}_{\{X^{b(\pi),\pi}_t \le b(\Pi^\pi_t)\}} \, dt \right],
		\end{displaymath} 
		which shows that $b$ solves \eqref{eq:NonlinearIntegralEquation}. 
		It remains to prove the uniqueness:
		By \eqref{eq:IntegrabilityCondition}, we obtain using Fubini's theorem that
		\begin{displaymath}
			V(x,\pi) =  \int_0^\infty e^{-rt} \mathbb{E}\left[\left( r - \mathcal{L} \right) g(X^{x,\pi}_t,\Pi^\pi_t) \mathbb{I}_D(X^{x,\pi}_t,\Pi^\pi_t) \right] \, dt,
		\end{displaymath}
		which means that we are in the situation of \cite[Sec. 2.2]{ChristensenCrocceMordeckiSalminen2019}. 
		Now, for all $a \in \mathcal{M}$ and $(x,\pi) \in (0,\infty) \times (0,1)$, let us define the stopping times
		\begin{displaymath}
			\gamma^{x,\pi}_{_{^{<a}}} := \inf\{t \ge 0| X^{x,\pi}_t < a(\Pi^\pi_t) \} \quad \text{and} \quad \gamma^{x,\pi}_{_{^{\ge a}}} := \inf\{t \ge 0 | X^{x,\pi}_t \ge a(\Pi^\pi_t)\}.
		\end{displaymath}	
		The claim that $b$ is the unique solution of \eqref{eq:NonlinearIntegralEquation} in $\mathcal{M}$ follows from \cite[Theorem 2.3]{ChristensenCrocceMordeckiSalminen2019}, if we prove that for all $a, \hat{a} \in \mathcal{M}$ the following implications hold true:
		\begin{itemize}
			\item[(i)] If for all $\pi \in (0,1)$ and $x \ge a(\pi) \vee \hat{a}(\pi)$ we have $\mathbb{P}\left( \lambda( \{ t \le \gamma^{x,\pi}_{_{^{<a}}} | X^{x,\pi}_t < \hat{a}(\Pi^\pi_t)\} ) = 0 \right) = 1$, then $ \hat{a}(\pi) \le a(\pi)$ for all $\pi \in (0,1)$.
			
			\item[(ii)] If for all $\pi \in (0,1)$ and $x \in (0,a(\pi))$ we have $\mathbb{P}\left( \lambda( \{ t \le \gamma^{x,\pi}_{_{^{\ge a}}} | X^{x,\pi}_t \ge \hat{a}(\Pi^\pi_t)\} ) = 0 \right) = 1$, then $\hat{a}(\pi) \ge a(\pi)$ for all $\pi \in (0,1)$.
		\end{itemize}
		To prove (i), we assume that there are $a,\hat{a} \in \mathcal{M}$ and $\hat{\pi} \in (0,1)$ such that $\hat{a}(\hat{\pi}) > a(\hat{\pi})$. 
		We define $O := \{(x,\pi) \in (0,\infty) \times (0,1) | \hat{a}(\pi) > x > a(\pi) \}$ and note that it is open since $a$ and $\hat{a}$ are continuous. 
		For each $\omega$ such that there is an $s < \gamma^{\hat{a}(\hat{\pi}), \hat{\pi}}_{_{^{<a}}}$ with $(X_s^{\hat{a}(\hat{\pi}), \hat{\pi}}, \Pi_s^{\hat{\pi}}) \in O$, we have 
		\begin{displaymath}
			\{t\le \gamma^{\hat{a}(\hat{\pi}), \hat{\pi}}_{_{^{<a}}} | X_t^{\hat{a}(\hat{\pi}), \hat{\pi}} < \hat{a}(\Pi_t^{\hat{\pi}}) \} \supseteq \{t \ge 0 | ( X_t^{\hat{a}(\hat{\pi}), \hat{\pi}},\Pi_t^{\hat{\pi}})\in O\} \supseteq [s, s+\tau_{O^c}^{X_s^{\hat{a}(\hat{\pi}), \hat{\pi}}, \Pi_s^{\hat{\pi}}}),
		\end{displaymath} 
		where $\tau^{x,\pi}_{O^c} := \inf\{ t \ge 0| (X^{x,\pi}_t,\Pi^\pi_t) \in O^c \}$ for all $(x,\pi) \in (0,\infty) \times [0,1]$. 
		Since $O$ is open, we have $\lambda([s, s+\tau_{O^c}^{X_s^{\hat{a}(\hat{\pi}), \hat{\pi}}, \Pi_s^{\hat{\pi}}}))>0$. 
		Hence, it suffices to establish that
		\begin{displaymath}
			\mathbb{P}\left( \{ \exists t < \gamma^{\hat{a}(\hat{\pi}), \hat{\pi}}_{_{^{<a}}} | ( X^{\hat{a}(\hat{\pi}), \hat{\pi}}_t, \Pi^{\hat{\pi}}_t ) \in O \} \right) > 0. 
		\end{displaymath}
		For this, we note that $\max\{\pi \in [0,1]: a(\pi) = \hat{a}(\hat{\pi})\} < \hat{\pi}$ since $a, \hat{a} \in \mathcal{M}$ are decreasing. 
		Hence, $\varepsilon:= \frac{1}{2}(\hat{\pi}-\max\{\pi \in [0,1]: a(\pi) = \hat{a}(\hat{\pi})\}) > 0$. 
		We set $\delta>0$ such that $\hat{a}(\hat{\pi}) - \delta > a(\hat{\pi}-\varepsilon)$ and
		\begin{displaymath}
			\frac{\hat{a}(\hat{\pi}) - \delta }{\hat{a}(\hat{\pi})} \ge \left( \frac{ (\hat{\pi}-\varepsilon) (1-\hat{\pi})}{(1-\hat{\pi}+\varepsilon) \hat{\pi} } \right)^{\frac{\sigma^2}{\mu_1-\mu_0}}.
		\end{displaymath}
		Further, we define $\tau^{\hat{a}(\hat{\pi}),\hat{\pi}}_{\hat{a}(\hat{\pi})-\delta} := \inf\{t \ge 0 | X^{\hat{a}(\hat{\pi}),\hat{\pi}}_t \le \hat{a}(\hat{\pi})-\delta\}$ and
		\begin{align*}
			t^* := \begin{cases}
				\frac{2}{ ( \sigma^2 - \mu_1 - \mu_0) } \ln \left( \frac{\hat{a}(\hat{\pi})}{\hat{a}(\hat{\pi})-\delta} \left( \frac{(\hat{\pi}-\varepsilon)(1-\hat{\pi})}{(1-\hat{\pi}+\varepsilon)\hat{\pi}} \right)^\frac{\sigma^2}{\mu_1-\mu_0} \right), & \text{if } \sigma^2-\mu_1-\mu_0 < 0 \\
				\frac{2}{ ( \sigma^2 - \mu_1 - \mu_0) }  \ln \left( \frac{\hat{a}(\hat{\pi})}{\hat{a}(\hat{\pi})-\delta} \right), & \text{if } \sigma^2-\mu_1-\mu_0 > 0 \\
				\infty, & \text{if } \sigma^2-\mu_1-\mu_0 = 0.
			\end{cases}
		\end{align*}
		By \cite[p.622, Eq. 2.0.2]{BorodinSalminen2002}, we have that $\mathbb{P}(\tau^{\hat{a}(\hat{\pi}),\hat{\pi}}_{\hat{a}(\hat{\pi})-\delta} \le t^*) > 0$. 
		Hence, if we prove that
		\begin{equation}
			\label{eq:ImplicationForUniqueness}
			\tau^{\hat{a}(\hat{\pi}),\hat{\pi}}_{\hat{a}(\hat{\pi})-\delta} \le t^* 
			\quad \Rightarrow \quad  \tau^{\hat{a}(\hat{\pi}),\hat{\pi}}_{\hat{a}(\hat{\pi})-\delta} < \gamma^{\hat{a}(\hat{\pi}),\hat{\pi}}_{_{^{<a}}} 
			\text{ and } (X^{\hat{a}(\hat{\pi}),\hat{\pi}}_{\tau^{\hat{a}(\hat{\pi}),\hat{\pi}}_{\hat{a}(\hat{\pi})-\delta}}, \Pi^{\hat{\pi}}_{\tau^{\hat{a}(\hat{\pi}),\hat{\pi}}_{\hat{a}(\hat{\pi})-\delta}} ) \in O
		\end{equation}
		the claim follows. 
		So, we assume $\tau^{\hat{a}(\hat{\pi}),\hat{\pi}}_{\hat{a}(\hat{\pi})-\delta} \le t^*$. 
		For $t \in [0,\infty)$,  Proposition~\ref{prop:PiExplicitSolution} implies that
		\begin{displaymath}
			 \Pi^{\hat{\pi}}_t \ge \hat{\pi}-\varepsilon 
			 \quad \Leftrightarrow \quad
			 \frac{X^{\hat{a}(\hat{\pi}), \hat{\pi}}_t}{\hat{a}(\hat{\pi})} e^{\frac{1}{2}(\sigma^2-\mu_1-\mu_0)t} \ge \left( \frac{ (\hat{\pi}-\varepsilon) (1-\hat{\pi})}{(1-\hat{\pi}+\varepsilon) \hat{\pi} } \right)^{\frac{\sigma^2}{\mu_1-\mu_0}}.
		\end{displaymath}
		For all $t \le \tau^{\hat{a}(\hat{\pi}),\hat{\pi}}_{\hat{a}(\hat{\pi})-\delta}$, we have
		 $X^{\hat{a}(\hat{\pi}), \hat{\pi}}_t \ge \hat{a}(\hat{\pi}) - \delta$, which implies by choice of $t^\ast$ that $\Pi^{\hat{\pi}}_t \ge \hat{\pi}-\varepsilon$.
		Since $a$ is decreasing, we have
		\begin{displaymath}
			a(\Pi^{\hat{\pi}}_t) \le a(\hat{\pi} - \varepsilon) 
			< \hat{a}(\hat{\pi}) - \delta 
			\le X^{\hat{a}(\hat{\pi}),\hat{\pi}}_t \quad \text{for all } t \le \tau^{\hat{a}(\hat{\pi}),\hat{\pi}}_{\hat{a}(\hat{\pi})-\delta},
		\end{displaymath}
		hence $\tau^{\hat{a}(\hat{\pi}),\hat{\pi}}_{\hat{a}(\hat{\pi})-\delta} < \gamma^{\hat{a}(\hat{\pi}),\hat{\pi}}_{_{^{<a}}}$. 
		Moreover, for $t \in [0,\infty)$, Proposition~\ref{prop:PiExplicitSolution} implies that
		\begin{displaymath}
			\Pi_t^{\hat{\pi}} \le \hat{\pi} 
			\quad \Leftrightarrow \quad 
			\frac{X^{\hat{a}(\hat{\pi}), \hat{\pi}}_t}{\hat{a}(\hat{\pi})} e^{\frac{1}{2}(\sigma^2-\mu_1-\mu_0)t} \le 1.
		\end{displaymath} 
		Hence, by choice of $t^\ast$ we have  $\Pi^{\hat{\pi}}_{\tau^{\hat{a}(\hat{\pi}),\hat{\pi}}_{\hat{a}(\hat{\pi})-\delta}} \le \hat{\pi}$. 
		All in all, it holds $\hat{\pi} - \varepsilon \le \Pi^{\hat{\pi}}_{\tau^{\hat{a}(\hat{\pi}),\hat{\pi}}_{\hat{a}(\hat{\pi})-\delta}} \le \hat{\pi}$. 
		Moreover, we have $X^{\hat{a}(\hat{\pi}),\hat{\pi}}_{\tau^{\hat{a}(\hat{\pi}),\hat{\pi}}_{\hat{a}(\hat{\pi})-\delta}} = \hat{a}(\hat{\pi}) - \delta$.
		Since $a$ is decreasing and by definition of $\delta$, we obtain
		\begin{displaymath}
			a(\Pi^{\hat{\pi}}_{\tau^{\hat{a}(\hat{\pi}),\hat{\pi}}_{\hat{a}(\hat{\pi})-\delta}}) \le a(\hat{\pi}-\varepsilon) < \hat{a}(\hat{\pi}) - \delta < \hat{a}(\hat{\pi}) \le \hat{a}(\Pi^{\hat{\pi}}_{\tau^{\hat{a}(\hat{\pi}),\hat{\pi}}_{\hat{a}(\hat{\pi})-\delta}}).
		\end{displaymath} 
		Hence, we proved $(X^{\hat{a}(\hat{\pi}),\hat{\pi}}_{\tau^{\hat{a}(\hat{\pi}),\hat{\pi}}_{\hat{a}(\hat{\pi})-\delta}}, \Pi^{\hat{\pi}}_{\tau^{\hat{a}(\hat{\pi}),\hat{\pi}}_{\hat{a}(\hat{\pi})-\delta}} ) \in O$.
		
		To prove (ii), we assume that there are $a,\hat{a} \in \mathcal{M}$ and $(\hat{x},\hat{\pi}) \in (0,\infty) \times (0,1)$ with $a(\hat{\pi}) > \hat{x} > \hat{a}(\hat{\pi})$. 
		We define $O := \{(x,\pi) \in (0,\infty) \times (0,1) | a(\pi) > x > \hat{a}(\pi) \}$ and note that it is open since $a$ and $\hat{a}$ are continuous. 
		Then we have 
		\begin{displaymath}
			\{t\le \gamma^{\hat{x}, \hat{\pi}}_{_{^{\ge a}}} | X_t^{\hat{x}, \hat{\pi}} \ge \hat{a}(\Pi_t^{\hat{\pi}}) \} \supseteq [0, \tau_{O^c}^{\hat{x},\hat{\pi}}),
		\end{displaymath} 
		where $\tau^{x,\pi}_{O^c} := \inf\{ t \ge 0| (X^{x,\pi}_t,\Pi^\pi_t) \in O^c \}$ for all $(x,\pi) \in (0,\infty) \times [0,1]$. 
		Since $O$ is open, we have $\lambda( [0, \tau_{O^c}^{\hat{x},\hat{\pi}}) )>0$. Hence, 
		\[		
			\mathbb{P}\left( \lambda( \{ t \le \gamma^{\hat{x},\hat{\pi}}_{_{^{\ge a}}} | X^{\hat{x},\hat{\pi}}_t \ge \hat{a}(\Pi^{\hat{\pi}}_t)\} ) = 0 \right) = 0. \qedhere
		\]
	\end{proof}	
	
\section{A Numerical Scheme for the Boundary Function}
	\label{sec:Numerics}
	
	Relying on the integral equation~\eqref{eq:NonlinearIntegralEquation}, we derive a numerically tractable fixed-point iteration scheme to compute the boundary function.
	We apply this scheme to compute and analyze the boundary function. 
	Moreover, we investigate the dependency of the value of information on the initial profit level and initial belief. 
	
	We established in Section~\ref{sec:Integral_boundary} that the boundary function is the unique solution of the integral equation~\eqref{eq:NonlinearIntegralEquation}.
	Using $V(b(\pi),\pi) =g(b(\pi),\pi)$, this implies that 
	\begin{displaymath}
		V(b(\pi),\pi) = g(b(\pi),\pi) - \mathbb{E} \left[ \int_0^\infty e^{-rt} \left( X^{b(\pi),\pi}_t - rI \right) \mathbb{I}_{\{ X^{b(\pi),\pi}_t \le b(\Pi^\pi_t)\}} \, dt \right].
	\end{displaymath} 
	Utilizing an independent exponential distributed random variable $\xi$ with parameter $r$, we rewrite the integral equation as
		\begin{displaymath}
		V(b(\pi),\pi) = g(b(\pi),\pi) - \mathbb{E} \left[ \tfrac{1}{r} \left( X^{b(\pi),\pi}_\xi - rI \right) \mathbb{I}_{\{ X^{b(\pi),\pi}_\xi \le b(\Pi^\pi_\xi)\}} \right].
	\end{displaymath}
	Rescaling this equation with $1/\left( \tfrac{1-\pi}{r-\mu_0} + \tfrac{\pi}{r-\mu_1}\right)$ implies that $b$ is a fixed point of the operator $\Psi$ that maps any function $b$ to the function $\Psi b$ with 
	\begin{displaymath}
		\Psi b (\pi) = b(\pi) - \frac{1}{\frac{1-\pi}{r-\mu_0}+\frac{\pi}{r-\mu_1}} \mathbb{E} \left[ \tfrac{1}{r} \left( X^{b(\pi),\pi}_\xi - rI \right) \mathbb{I}_{\{ X^{b(\pi),\pi}_\xi \le b(\Pi^\pi_\xi)\}} \right], \quad \pi \in [0,1].
	\end{displaymath} 
	Hence, as in \cite{ ChristensenSalminen2018, DammannFerrari2022,DetempleKitapbayev2020}, it is natural to consider the fixed point iteration
	\begin{displaymath}
		b^{(n)} = \Psi(b^{(n-1)})
	\end{displaymath} 
	for a suitable chosen initial boundary $b^{(0)} \in \mathcal{M}$. 
	Here, the lower bound $\underline{b}$ introduced in Theorem~\ref{thm:PropertiesBoundaryFunction} is a natural candidate.
	Using this lower bound and approximating the expectation by using the Monte Carlo method, also in our setting this scheme turns out to converge fast to an approximate solution of \eqref{eq:NonlinearIntegralEquation}.
	
	Utilizing this iteration scheme, we compute the boundary functions for several parameter constellations. 
	Figure~\ref{fig:BoundaryFunction} depicts the boundary function $b$, the lower bound $\underline{b}$ and stopping region for four different parameter constellations. 
	We observe that the boundary functions all satisfy the properties described in Theorem~\ref{thm:PropertiesBoundaryFunction}, i.e., they are decreasing continuous functions with values in $[x_1^*,x_0^*]$  that lie above the lower bound.
	For some cases, we observe that the lower bound and the boundary are close to each other, for other cases there is a significant gap. 
	This behavior cannot be explained by varying one of the parameters or the signal-to-noise ratio and is left for future research. 
	
	Moreover, Figure~\ref{fig:BoundaryFunction} shows that the greater the signal-to-noise ratio, the more the boundary function and the lower bound are pushed to the upper right corner. 
	This behavior is robust and has been observed in all numerical experiments carried out by the authors. 
	Intuitively, this is explained as follows: 
	If the signal-to-noise ratio is large, then the belief process converges faster to $0$ or $1$, respectively, which means that the agent learns faster about the true state of the unknown random variable $\theta$. 
	Hence, waiting has larger benefits in these situations, which gives a larger continuation region.
		
	\begin{figure}[ht]
		\begin{minipage}[t]{0.495\textwidth}
			\includegraphics[width=0.9\textwidth]{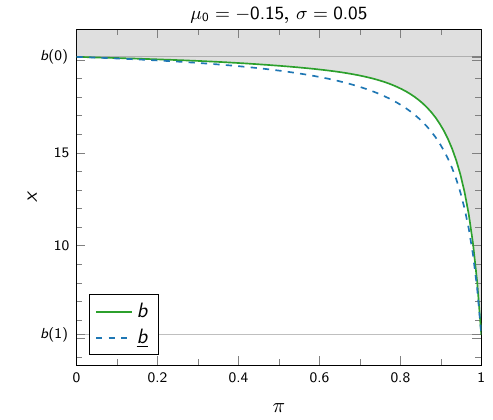}
		\end{minipage}
		\begin{minipage}[t]{0.495\textwidth}
			\includegraphics[width=0.9\textwidth]{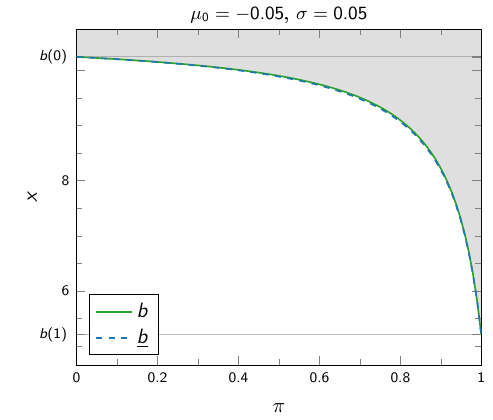}
		\end{minipage}
		\begin{minipage}[t]{0.495\textwidth}
			\includegraphics[width=0.9\textwidth]{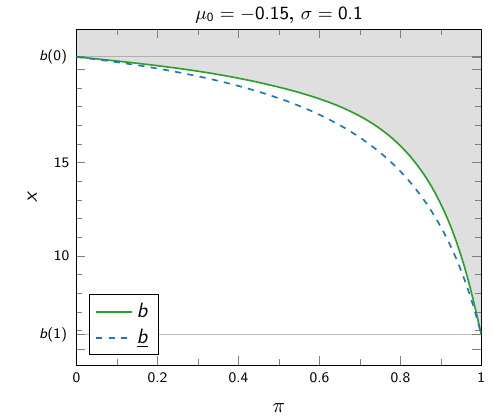}
		\end{minipage}
		\begin{minipage}[t]{0.495\textwidth}
			\includegraphics[width=0.9\textwidth]{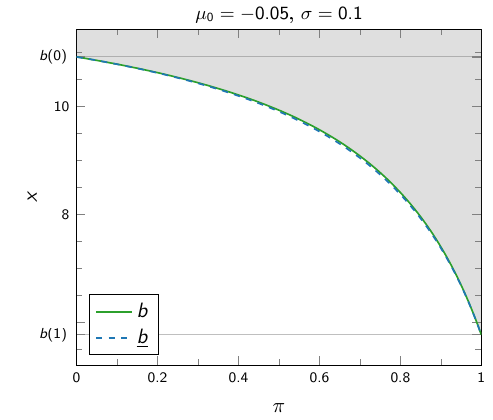}
		\end{minipage}
		\caption{Boundary function, lower bound and stopping set (gray) for different parameters; $r=0.05$, $\mu_1=0.03$ and $I=100$ fixed}
		\label{fig:BoundaryFunction}
	\end{figure}
	
	As usual, we are also interested in the value of information, i.e., the additional value generated by full knowledge of the model, in our case the knowledge about $\theta$. 
	In our setting, the value of information is given by
	\begin{displaymath}
		\Delta(x,\pi) := \overline{V}(x,\pi) - V(x,\pi),
	\end{displaymath} 
	where $\overline{V}(x,\pi) := (1-\pi) V(x,0) + \pi V(x,1)$ is the value of the problem with full information, that is the agent knows at time $0-$ the true value of $\theta$ and hence uses the optimal strategy for the corresponding problem with full information. By Corollary~\ref{cor:UpperBoundValueFct}, the value of information is always non-negative.
	It is immediate from our results that $V(x,\pi) = \overline{V}(x,\pi)$ for $\pi \in [0,1]$ and $x \ge b(0)$, which means that when the profit level exceeds $b(0)$, then there is no value of information. 
	The reason is that then the agents stop immediately in either case.
	Hence, it suffices to consider profit levels ranging from $0$ to $b(0)$ in our numerical experiments. 
	In all our numerical experiments we observe that the value of information is in general within or below the single-digit percentage range compared to the investment costs. 
	Moreover, we observe that a larger signal-to-noise ratio leads to a smaller value of information. 
	This is again intuitive, since a larger signal-to-noise ratio implies faster learning of the true state of $\theta$ and hence, more informed decisions by the uninformed agent.
	
	Figure~\ref{fig:ValueOfInfo} depicts the value of information $\Delta(\cdot,\pi)$ for fixed initial beliefs $\pi$ and for the same parameter constellations as in Figure~\ref{fig:BoundaryFunction}. 	
	We observe that the value of information $\Delta(\cdot,\pi)$ is close to zero near the origin, then it increases until reaching its peak within $(b(1),b(\pi))$, and decreases thereafter until it reaches $0$ at $b(0)$.
	The profit level associated with the peak is decreasing with $\pi$ and the largest peak is observed for a medium $\pi$.
	The reason for this behavior, which is again robust for all our numerical experiments, is that the value of information has two components: 
	One the hand, it might happen that the uniformed agent stops, whereas the informed agent would continue. 
	This happens for $x \in (b(\pi),b(0))$ and $\theta=0$. 
	On the other hand, it might happen that the uniformed agent continues, whereas the informed agent would stop. 
	This happens for $x \in (b(1),b(\pi))$ and $\theta=1$. 
	Since these losses are larger, the larger is the distance of $x$ and $b(0)$ or $b(1)$, respectively, and the larger is the probability for these cases to happen (which is $1-\pi$ or $\pi$, respectively), this explains the form of the value of information. \\

	\begin{figure}[ht]
		\begin{minipage}[t]{0.495\textwidth}
			\includegraphics[width=0.99\textwidth]{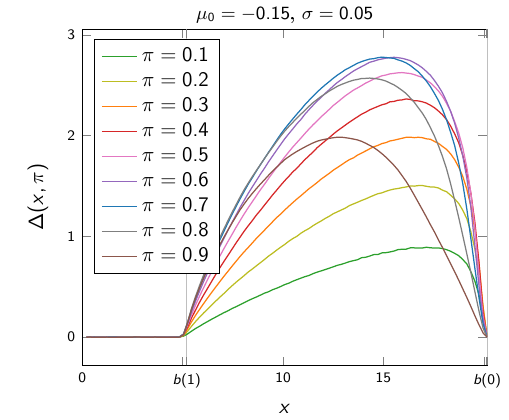}
		\end{minipage}
		\begin{minipage}[t]{0.495\textwidth}
			\includegraphics[width=0.99\textwidth]{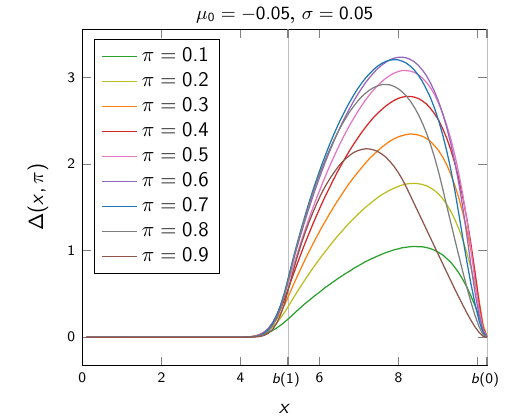}
		\end{minipage}
		\begin{minipage}[t]{0.495\textwidth}
			\includegraphics[width=0.99\textwidth]{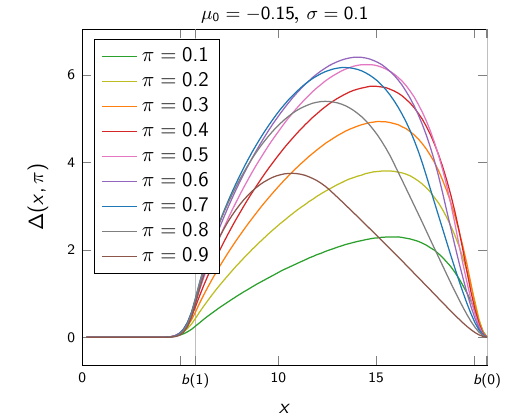}
		\end{minipage}
		\begin{minipage}[t]{0.495\textwidth}
			\includegraphics[width=0.99\textwidth]{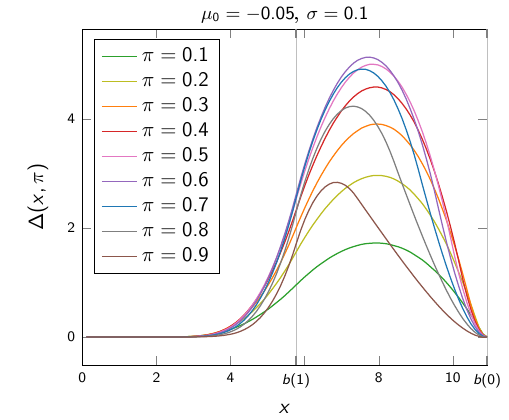}
		\end{minipage}
		\caption{Value of information $\Delta$ for different parameters; $r=0.05$, $\mu_1=0.03$ and $I=100$ fixed}
		\label{fig:ValueOfInfo}
	\end{figure}	
	
	\textbf{Acknowledgements}
	The authors thank the DFG for their support within RTG 2126 "Algorithmic Optimization".
	
	\textbf{Declarations of interest:} None.

	\bibliographystyle{plain}
	\bibliography{literature_investment}
	
	\appendix
\section{Appendix}
	\label{sec:appendix}
	
	\begin{proof}[Proof of Proposition~\ref{prop:existence_first_entry}]
		We note that $0 \le X_t^{x,\pi} \le X_t^{x,1}$ for all $t \ge 0$~a.s.\ and, since $\mu_0 < \mu_1$, we have $\tfrac{1- \hat{\pi}}{r-\mu_0} + \tfrac{\hat{\pi}}{r-\mu_1} \le \tfrac{1}{r-\mu_1}$ for all $\hat{\pi} \in [0,1]$. 
		Hence,
		\begin{displaymath}
			\mathbb{E} \left[ \sup_{t \geq 0} \left| e^{-r(s+t)} g(X^{x,\pi}_{t},\Pi^\pi_{t}) \right| \right]
			\le \frac{e^{-rs}}{r-\mu_1} \mathbb{E} \left[ \sup_{t \geq 0}  e^{-rt} X^{x,1}_t  \right] + e^{-rs} I.
		\end{displaymath}
		This expectation is finite since for $\nu := \frac{\mu_1-r}{\sigma^2} - \frac{1}{2}$ (note that $2 \nu +1 < 0$) we obtain by \cite[p.622, Eq. 2.0.2(1)]{BorodinSalminen2002} that 
		\begin{align*}
			\mathbb{E} \left[ \sup_{t \geq 0}  e^{-rt} X^{x,1}_t  \right]
			&= \int_0^\infty \mathbb{P} \left(\sup_{t \geq 0} e^{-rt} X^{x,1}_t \ge y\right) \, dy 
			= \int_0^x 1 \, dy + \int_x^\infty \left(\frac{x}{y}\right)^{-2\nu} \, dy \\
			&= x + x^{-2\nu} \frac{1}{2\nu + 1} \left[ y^{2\nu+1} \right]_{y = x}^\infty 
			= x \left(1+ \frac{\sigma^2}{2(r-\mu_1)} \right) < \infty.
		\end{align*}
		Now, for $(s,x,\pi) \in [0, \infty) \times (0,\infty) \times[0,1]$ we consider the auxiliary optimal stopping problem 
		\begin{displaymath}
			\widetilde{V}(s,x,\pi) := \sup_{\tau \in \mathcal{T}^{x,\pi}} \mathbb{E} \left[ \tilde{g} \left(s+\tau, X^{x,\pi}_\tau, \Pi^\pi_\tau \right) \right]
		\end{displaymath} 
		with payoff function $\tilde{g}(t,\hat{x},\hat{\pi}) := e^{-rt} g(\hat{x},\hat{\pi})$. 
		Then, $(s+t, X^{x,\pi}_t,\Pi^\pi_t)_{t \ge 0}$ is an $\mathcal{F}^{X^{x,\pi}}$-adapted Markov process with continuous paths, satisfying $\lim_{t \to 0} \tilde{g}(s+t,X^{x,\pi}_t, \Pi^\pi_t) = \tilde{g}(s,x,\pi)$~a.s.
		Since we established before that $\mathbb{E} \left[ \sup_{t \geq 0} \left| \tilde{g}(s+t, X^{x,\pi}_{t},\Pi^\pi_{t}) \right| \right]$ is finite, we are in the setting of \cite[Ch. 3, Theorem 3(2)]{Shiryaev1978}.
		Hence, $\inf \{t \ge 0: (s+t,X_t^{x,\pi},\Pi^\pi_t) \in \widetilde{D}\}$ with
		\begin{displaymath}
			\widetilde{D} = \left\{ (t,\hat{x},\hat{\pi}) \in [0,\infty) \times (0,\infty) \times [0,1] : \widetilde{V}(t,\hat{x}, \hat{\pi}) = \widetilde{g}(t,\hat{x},\hat{\pi}) \right\}
		\end{displaymath} 
		is an optimal stopping time. 
		To prove the claim, it remains to show that $\widetilde{D} = [0, \infty) \times D$. 
		Let $(\hat{t},\hat{x}, \hat{\pi}) \in \widetilde{D}$. 
		Then we obtain for any $t \ge 0$ that
		\begin{align*}
			\widetilde{V}(t,\hat{x},\hat{\pi}) 
			&= \sup_{\tau \in \mathcal{T}^{\hat{x}, \hat{\pi}}} \mathbb{E} \left[ e^{-r(t+\tau)} g \left( X^{\hat{x},\hat{\pi}}_\tau, \Pi^{\hat{\pi}}_\tau \right) \right] 
			= e^{-r(t-\hat{t})} \sup_{\tau \in \mathcal{T}^{\hat{x},\hat{\pi}}} \mathbb{E} \left[ e^{-r(\hat{t} +\tau)} g \left( X^{\hat{x},\hat{\pi}}_\tau, \Pi^{\hat{\pi}}_\tau \right) \right] \\
			&= e^{-r(t-\hat{t})} \widetilde{V}(\hat{t},\hat{x},\hat{\pi}) 
			= e^{-r(t-\hat{t})} \widetilde{g}(\hat{t},\hat{x},\hat{\pi}) 
			= \widetilde{g}(t,\hat{x},\hat{\pi}),
		\end{align*} 
		which implies $(t, \hat{x}, \hat{\pi}) \in \tilde{D}$.
		Noting that $V(\hat{x},\hat{\pi}) = \widetilde{V}(0,\hat{x}, \hat{\pi})$ and $g(\hat{x},\hat{\pi}) = \widetilde{g}(0,\hat{x}, \hat{\pi})$, we obtain 
		\begin{displaymath}
			\widetilde{D} 
			= \left\{ (t,\hat{x},\hat{\pi}) \in  [0,\infty) \times (0,\infty) \times [0,1] : V(\hat{x}, \hat{\pi}) = g(\hat{x},\hat{\pi}) \right\}
			= [0,\infty) \times D. \qedhere
		\end{displaymath}
	\end{proof}

	\begin{proof}[Proof of Proposition \ref{prop:PiExplicitSolution}]
		Fix $t \ge 0$. For $i \in \{0,1\}$, we set 
		\begin{align*}
			E_i &:= \exp\left( \left(\mu_i - \frac{1}{2} \sigma^2 \right) \int_0^t \frac{1}{\sigma^2} \, dS^\pi_u - \frac{1}{2}\left(\mu_i - \frac{1}{2} \sigma^2 \right)^2 \int_0^t \frac{1}{\sigma^2} \, du \right) \\
			&= \exp \left( \frac{\mu_i-\frac{1}{2}\sigma^2}{\sigma^2} S^\pi_t - \frac{(\mu_i-\frac{1}{2}\sigma^2)^2}{2\sigma^2} t \right)
		\end{align*}
		and obtain by \cite[Eq. (5)]{Wonham1964} that
		\begin{displaymath}
			\Pi^\pi_t = \frac{ \pi E_1 } {(1-\pi) E_0 + \pi E_1} = \frac{\pi E_1 E_0^{-1}}{1-\pi + \pi E_1 E_0^{-1}}.
		\end{displaymath}
		Since
		\begin{displaymath}
			E_1E_0^{-1} 
			= \exp \left( \frac{\mu_1-\mu_0}{\sigma^2} S^\pi_t + \frac{\mu_1-\mu_0}{\sigma^2} \frac{1}{2} \left(\sigma^2-\mu_1-\mu_0\right) t \right) 
			=\left( \frac{X_t^{x,\pi}}{x} \exp \left( \frac{1}{2} (\sigma^2 - \mu_1 - \mu_0) t \right) \right)^{\frac{\mu_1-\mu_0}{\sigma^2}},
		\end{displaymath} 
		we obtain $\Pi_t^\pi = f(t,x, \pi, X_t^{x,\pi})$ as claimed.
	\end{proof}
	
	\begin{lemma}
		\label{lem:beta_0_beta_1}
		It holds that $\beta_0 > \beta_1$ and $x_0^\ast > x_1^\ast \ge rI$.
	\end{lemma}
	
	\begin{proof}
		By equation \eqref{eq:BetaQuadraticEquation}, we have $\mu_i = \frac{r}{\beta_i} - \frac{1}{2} \sigma^2 (\beta_i -1 )$
		for $i \in \{0,1\}$, which implies
		\begin{displaymath}
			0 
			< \mu_1-\mu_0 
			= \left( \frac{r}{\beta_0 \beta_1} + \frac{1}{2} \sigma^2 \right) (\beta_0 - \beta_1).
		\end{displaymath}
		Since $\frac{r}{\beta_0 \beta_1} + \frac{1}{2} \sigma^2 > 0$, we obtain $\beta_0 > \beta_1$. 
		Moreover, we have
		\begin{displaymath}
			x^*_i = \frac{\beta_i}{\beta_i-1} (r-\mu_i) I
			= \frac{\beta_i}{\beta_i-1} \left(r- \frac{r}{\beta_i} + \frac{1}{2} \sigma^2 (\beta_i -1 ) \right) I
			= \left( r + \frac{1}{2}\sigma^2 \beta_i \right) I
		\end{displaymath}
		for $i \in \{0,1\}$. 
		Hence, $\beta_0 > \beta_1>1$ implies $x^*_0 > x^*_1\ge rI$.
	\end{proof}
	
	\begin{lemma}
		\label{lem:LowerBoundDecreasing}
		The function $\underline{b}$ from \eqref{eq:LowerBoundForBoundaryFunction} is decreasing with $\underline{b}([0,1]) = [x_1^*,x_0^*]$.
	\end{lemma}
	
	\begin{proof}
		We fix $\pi \in [0,1]$ and compute
		\begin{align*}
			\frac{\underline{b}(\pi)-b(1)}{b(0) - b(1)}
			&= \frac{ \left[ (1-\pi) \beta_0 + \pi \beta_1 \right] - \beta_1 \left[ (1-\pi)\frac{\beta_0-1}{\beta_1-1}\frac{r-\mu_1}{r-\mu_0} + \pi \right] }{ \left[ (1-\pi)\frac{\beta_0-1}{r-\mu_0} + \pi \frac{\beta_1-1}{r-\mu_1} \right] \left[ \frac{\beta_0}{\beta_0-1} (r-\mu_0) - \frac{\beta_1}{\beta_1-1} (r-\mu_1) \right]} \\ 
			&= \frac{ (1-\pi) \left[ \beta_0 - \beta_1 \frac{\beta_0-1}{\beta_1-1}\frac{r-\mu_1}{r-\mu_0} \right] }{ (1-\pi) \left[ \beta_0 - \beta_1 \frac{\beta_0-1}{\beta_1-1}\frac{r-\mu_1}{r-\mu_0} \right] + \pi \left[ \beta_0 \frac{\beta_1-1}{\beta_0-1}\frac{r-\mu_0}{r-\mu_1} - \beta_1 \right] } = \frac{ 1-\pi }{ 1-\pi + \pi   \frac{\beta_1-1}{\beta_0-1}\frac{r-\mu_0}{r-\mu_1} }.
		\end{align*}
		From \eqref{eq:BetaQuadraticEquation}, we have for $i \in \{0,1\}$ that
		$- \mu_i \beta_i = \frac{1}{2} \sigma^2 \beta_i (\beta_i-1) - r$ and hence
		\begin{displaymath}
			\frac{r-\mu_i}{\beta_i-1} 
			= \frac{r \beta_i -\mu_i \beta_i}{(\beta_i-1)\beta_i}
			= \frac{\frac{1}{2}  \sigma^2 \beta_i (\beta_i-1) + r (\beta_i-1)}{(\beta_i-1)\beta_i}
			= \frac{1}{2}\sigma^2 + \frac{r}{\beta_i}.
		\end{displaymath}
		Since $\beta_0 > \beta_1$, it holds $\frac{\beta_1-1}{\beta_0-1}\frac{r-\mu_0}{r-\mu_1} = \frac{\frac{1}{2}\sigma^2 + \frac{r}{\beta_0}}{\frac{1}{2}\sigma^2 + \frac{r}{\beta_1}} < 1$. 
		Thus, $\underline{b}$ is decreasing.
		Noting that $\underline{b}$ is continuous and that $\underline{b}(0) = x^*_0$ and $\underline{b}(1) = x^*_1$, we finally obtain $\underline{b}([0,1]) = [x^*_1,x^*_0]$.
	\end{proof}
	
\end{document}